\documentclass[leqno,english]{amsart}

\usepackage{hyperref}   
\usepackage[square,numbers]{natbib}     
\usepackage{amsthm}     
\usepackage{amsmath}    
\usepackage{amssymb}    
\usepackage{amsxtra}    
\usepackage{eucal}      
\usepackage[bbgreekl]{mathbbol}   
\usepackage[all]{xy}    

\theoremstyle{plain}

\swapnumbers                        
\newtheorem{prop}[subsubsection]{Proposition}

\newtheorem{thm}[subsubsection]{Theorem}

\numberwithin{equation}{subsubsection}

\title{The extended rational homotopy theory of operads}

\date{April 2, 2018}

\author{Benoit Fresse}
\address{Univ. Lille, CNRS, UMR 8524 - Laboratoire Paul Painlev\'e, F-59000 Lille, France}
\email{Benoit.Fresse@math.univ-lille1.fr}

\thanks{This research is supported in part by Labex ANR-11-LABX-0007-01 ``CEMPI''.
The author is grateful to Thomas Willwacher for motivating exchanges
which are at the origin of this work}


\DeclareMathOperator{\kk}{\mathbb{k}}   
\DeclareMathOperator{\NN}{\mathbb{N}}   
\DeclareMathOperator{\QQ}{\mathbb{Q}}   
\DeclareMathOperator{\RR}{\mathbb{R}}   
\DeclareMathOperator{\DD}{\mathbb{D}}   

\DeclareMathOperator{\CCat}{\mathcal{C}}            
\DeclareMathOperator{\Mod}{\mathcal{M}\mathit{od}}  
\DeclareMathOperator{\Top}{\mathcal{T}\mathit{op}}  
\DeclareMathOperator{\Simp}{\mathit{s}\mathcal{S}\mathit{et}}   
\DeclareMathOperator{\ComCat}{\mathcal{C}\mathit{om}}   

\DeclareMathOperator{\Hopf}{\mathcal{H}\mathit{opf}}    
\DeclareMathOperator{\Seq}{\mathcal{S}\mathit{eq}}  
\DeclareMathOperator{\Op}{\mathcal{O}\mathit{p}}    

\DeclareMathOperator{\dg}{\mathit{dg}}      
\DeclareMathOperator{\simp}{\mathit{s}}     

\DeclareMathOperator{\Mor}{\mathtt{Mor}}    
\DeclareMathOperator{\Map}{\mathtt{Map}}    

\DeclareMathOperator{\Id}{\mathit{Id}}      
\DeclareMathOperator{\pt}{\mathit{pt}}      
\DeclareMathOperator{\unit}{\mathbb{1}}     

\DeclareMathOperator{\FreeOp}{\mathbb{F}}   
\DeclareMathOperator{\Sym}{\mathbb{S}}      

\DeclareMathOperator{\Res}{\mathtt{Res}}        

\DeclareMathOperator{\SO}{\mathtt{SO}}
\DeclareMathOperator{\Sing}{\mathtt{Sing}}      
\DeclareMathOperator{\DGB}{\mathtt{B}}          
\DeclareMathOperator{\DGG}{\mathtt{G}}
\DeclareMathOperator{\DGH}{\mathtt{H}}          
\DeclareMathOperator{\DGL}{\mathtt{L}}
\DeclareMathOperator{\DGN}{\mathtt{N}}          
\DeclareMathOperator{\DGR}{\mathtt{R}}
\DeclareMathOperator{\DGOmega}{\mathtt{\Omega}}          

\DeclareMathOperator{\Tot}{\mathtt{Tot}}        

\DeclareMathAlphabet{\mathsfit}{OT1}{cmss}{m}{sl}   
\DeclareMathOperator{\AOp}{\mathsfit{A}}
\DeclareMathOperator{\BOp}{\mathsfit{B}}
\DeclareMathOperator{\COp}{\mathsfit{C}}
\DeclareMathOperator{\DOp}{\mathsfit{D}}        
\DeclareMathOperator{\IOp}{\mathsfit{I}}        
\DeclareMathOperator{\KOp}{\mathsfit{K}}
\DeclareMathOperator{\LOp}{\mathsfit{L}}
\DeclareMathOperator{\MOp}{\mathsfit{M}}
\DeclareMathOperator{\NOp}{\mathsfit{N}}
\DeclareMathOperator{\POp}{\mathsfit{P}}        
\DeclareMathOperator{\QOp}{\mathsfit{Q}}        
\DeclareMathOperator{\ROp}{\mathsfit{R}}        
\DeclareMathOperator{\SOp}{\mathsfit{S}}

\DeclareMathOperator{\ComOp}{\mathsfit{Com}}    

\DeclareMathOperator{\rset}{\underline{\mathsf{r}}}     

\DeclareMathOperator{\stree}{\underline{\mathsf{S}}}    
\DeclareMathOperator{\ttree}{\underline{\mathsf{T}}}    
\DeclareMathOperator{\thetatree}{\underline{\mathsf{\Theta}}}
\DeclareMathOperator{\unittree}{\underline{\mathsf{I}}}

\begin{document}

\begin{abstract}
In this paper, we set up a rational homotopy theory for operads in simplicial sets whose term of arity one
is not necessarily reduced to an operadic unit,
extending results obtained by the author in the book ``Homotopy of operads
and Grothendieck-Teichm\"uller groups''.
In short, we prove that the rational homotopy type of such an operad is determined
by a cooperad in cochain differential graded algebras (a cochain Hopf dg-cooperad for short)
as soon as the Sullivan rational homotopy theory works
for the spaces underlying our operad (e.g. when these spaces
are connected, nilpotent,
and have finite type rational cohomology groups).
\end{abstract}

\maketitle

\begin{center}\emph{This paper is dedicated to Tornike Kadeishvili for his 70th anniversary}\end{center}

\section*{Introduction}

In~\cite{FresseBook}, the author proved that the Sullivan model for the rational homotopy theory of spaces
admits an extension to the category of operads in simplicial sets $\POp$ whose term of arity one
is reduced to a one-point set $\POp(1) = *$.
The goal of this paper is to explain the definition of an extension of this operadic Sullivan model for operads
that do not necessarily reduce to a one-point set in arity one.
This extension enables us to apply our rational homotopy theory to the framed little discs operads $\DOp_n^{fr}$,
for which we have $\DOp_n^{fr}(1)\sim\SO(n)$, where $\SO(n)$ denotes the group of rotations
of the Euclidean space $\RR^n$. To be more precise, we then pass to real coefficients,
and we prove, by relying on the formality result of~\cite{KhoroshkinWillwacher}, that the real homotopy type of the operad $\DOp_n^{fr}$
admits a model which we get by taking a realization $<\DGH^*(\DOp_n^{fr})>$ (in the sense of Sullivan)
of the cohomology with real coefficients of this operad $\DGH^*(\DOp_n^{fr}) = \DGH^*(\DOp_n^{fr},\RR)$
when $n$ is even.

Briefly recall that the Sullivan model of operads which we consider in~\cite{FresseBook}
consists of cooperads in the category of commutative cochain dg-algebras,
where a cooperad refers to a structure
which is dual to an operad in the categorical sense.
To summarize, a cooperad in commutative cochain dg-algebras $\AOp$ (a Hopf cochain dg-cooperad for short)
consists of a collection of commutative cochain dg-algebras $\AOp(r)$, $r>0$,
so that the symmetric group on $r$ letters $\Sigma_r$ acts on $\AOp(r)$ through dg-algebra isomorphisms,
for each $r>0$,
and we have composition coproducts $\circ_i^*: \AOp(k+l-1)\rightarrow\AOp(k)\otimes\AOp(l)$, defined in the category of dg-algebras,
which satisfy the dual of the equivariance, associativity, and unit relations
of the composition products of an operad.
In what follows, we generally take the field of rational numbers as a ground field (we only pass to real coefficients
when we address the applications to the framed little discs operads),
and we accordingly assume, to be precise, that our commutative cochain dg-algebras
are defined over this field $\kk = \QQ$.

In~\cite{FresseBook}, in order to define a rational model of operads satisfying $\POp(1) = *$,
we restrict ourselves to Hopf cochain dg-cooperads satisfying $\AOp(1) = \QQ$.
This condition ensures that our cooperads are conilpotent, in the sense that we can not iterate the composition coproducts
indefinitely. This property simplifies the definition of our model.
In particular, we prove in~\cite{FresseBook} that the category of Hopf cochain dg-cooperads
such that $\AOp(1) = \QQ$ inherits a model category structure with the quasi-isomorphisms of Hopf cochain dg-cooperads
as class of weak-equivalences.

Let $\DGOmega^*: \Simp^{op}\rightarrow\dg^*\ComCat$ now denote the functor of piecewise linear differential forms
from the category of simplicial sets $\Simp$ to the category of commutative cochain dg-algebras $\dg^*\ComCat$,
such as defined by Sullivan in~\cite{Sullivan}. (In what follows, we also refer to~\cite[\S II.7]{FresseBook} for a general account of the definition
of the Sullivan model.)
Let $\DGG: \dg^*\ComCat\rightarrow\Simp^{op}$ denote the left adjoint of this functor, which is defined
by $\DGG(A) = \Mor_{\dg^*\ComCat}(A,\DGOmega^*(\Delta^{\bullet}))$,
for any $A\in\dg^*\ComCat$, where $\Delta^{\bullet}$ denotes the cosimplicial object formed by the collection of the simplices $\Delta^n$, $n\in\NN$,
in the category of simplicial sets $\Simp$.
Briefly recall that this adjunction $\DGG: \dg^*\ComCat\rightleftarrows\Simp^{op} :\DGOmega^*$ defines a Quillen adjunction,
and that the image of a connected simplicial set $X\in\Simp$
under the composite of the derived functors of this Quillen pair is identified with the rationalization of the space $X$
under mild finiteness and nilpotence assumptions.
To be explicit, we have the identity $X^{\QQ} = \DGL\DGG(\DGOmega^*(X))$, where we use the notation $\DGL\DGG(-)$
for the left derived functor of the functor $\DGG: \dg^*\ComCat\rightarrow\Simp^{op}$.
(We do not need to derive the Sullivan functor because all simplicial sets form cofibrant objects in the category of simplicial sets.)

In~\cite{FresseBook}, we mainly observe that the adjunction $\DGG: \dg^*\ComCat\rightleftarrows\Simp^{op} :\DGOmega^*$
admits an extension $\DGG: \dg^*\Hopf\Op_{01}^c\rightleftarrows\Simp\Op_{\varnothing 1}^{op} :\DGOmega^*_{\sharp}$,
where $\Simp\Op_{\varnothing 1}$ denotes the category of operads in simplicial sets such that $\POp(1) = *$
and $\dg^*\Hopf\Op_{01}^c$ denotes the category of Hopf cochain dg-cooperads
such that $\AOp(1) = \QQ$.
The functor $\DGG: \dg^*\Hopf\Op_{01}^c\rightarrow\Simp\Op_{\varnothing 1}^{op}$ is defined by an arity-wise
application of the left adjoint of the Sullivan functor $\DGG: \dg^*\ComCat\rightarrow\Simp^{op}$.
We explicitly have $\DGG(\AOp)(r) = \DGG(\AOp(r))$, for each $r>0$,
for any $\AOp\in\dg^*\Hopf\Op_{01}^c$.
In fact, the Sullivan functor $\DGOmega^*: \Simp^{op}\rightarrow\dg^*\ComCat$ does not preserve cooperad structures,
but we just prove in~\cite{FresseBook} that we can use the adjoint functor lifting theorem to fix this problem
in order to define the right adjoint functor $\DGOmega^*_{\sharp}: \Simp\Op_{\varnothing 1}^{op}\rightarrow\dg^*\Hopf\Op_{01}^c$
of our adjunction relation between the category of operads $\Simp\Op_{\varnothing 1}$
and the category of Hopf cochain dg-cooperad $\dg^*\Hopf\Op_{01}^c$. Then we set $\POp^{\QQ} = \DGL\DGG(\DGOmega^*_{\sharp}(\POp))$
in order to define the rationalization of an operad in simplicial sets $\POp\in\Simp\Op_{\varnothing 1}$.
In this construction, we just need to assume that the operad $\POp$ is cofibrant (we take a cofibrant resolution of $\POp$ otherwise).
In~\cite{FresseBook}, we check that, under mild finiteness assumptions,
we have an arity-wise weak-equivalence $\DGOmega^*_{\sharp}(\POp)(r)\sim\DGOmega^*(\POp(r))$
when we consider the collection of commutative cochain dg-algebras $\DGOmega^*_{\sharp}(\POp)(r)$
which underlie our operadic enhancement of the Sullivan model $\DGOmega^*_{\sharp}(\POp)$.
This result implies that our operadic rationalization functor $\POp^{\QQ}$
reduces to the Sullivan rationalization of spaces
arity-wise. We explicitly have a weak-equivalence of simplicial sets $\POp^{\QQ}(r)\sim\POp(r)^{\QQ}$
in each arity $r>0$, where we again consider the Sullivan rationalization $X^{\QQ}$
of the simplicial set $X = \POp(r)$
on the right-hand side.

To extend this theory to the category of operads that do not necessarily reduce to a one-point set in arity one,
we consider general conilpotent Hopf cochain dg-cooperads $\AOp$,
for which the term of arity one $\AOp(1)$ does not necessarily reduce to the ground field $\QQ$,
but where the composition coproducts of a given element vanish after enough iterations.
We make this condition more precise in the first section of the paper.

We quickly check that the definition of our model structure on Hopf cochain dg-cooperads
extends to this category, as well as the definition
of our operadic upgrading of the Sullivan functor $\DGOmega^*_{\sharp}: \POp\mapsto\DGOmega^*_{\sharp}(\POp)$
for the operads $\POp$ which do not necessarily reduce to a one-point set
in arity one.
We again prove that we have a weak-equivalence of commutative cochain dg-algebras in each arity $\DGOmega^*_{\sharp}(\POp)(r)\sim\DGOmega^*(\POp(r))$,
but we need new argument lines in order to establish this result.
Indeed, the idea of~\cite{FresseBook} is to reduce the verification to the case of a free operad.
But in order to prove that we have arity-wise weak-equivalences $\DGOmega^*_{\sharp}(\POp)(r)\sim\DGOmega^*(\POp(r))$
we need to assume that the simplicial set $\POp(1)$ is connected
when this simplicial set is not reduced to a point,
and this condition is not fulfilled by the free operad.
The main idea is to use an operadic analogue of the James construction, rather than the usual free operad,
in order to adapt our arguments.
This verification represents the crux of our construction and is explained with full details in the second section
of the paper. Then we give a brief survey of the applications
of our construction to the definition
of the model for the rational homotopy of operads.
We just briefly check that our results
extend without change
as soon as we have the arity-wise weak-equivalences $\DGOmega^*_{\sharp}(\POp)(r)\sim\DGOmega^*(\POp(r))$
for our operadic upgrading of the Sullivan functor $\DGOmega^*_{\sharp}: \POp\mapsto\DGOmega^*_{\sharp}(\POp)$.
We devote the third section of the paper to this survey.

We have not been precise about the component of arity zero of our operads so far. We generally assume $\POp(0) = \varnothing$.
But in~\cite{FresseBook} we explain that we can model the structure of an operad $\POp_+$ satisfying $\POp_+(0) = *$
by providing the operad $\POp$ such that $\POp(0) = \varnothing$ and $\POp(r) = \POp_+(r)$ for $r>0$
with extra restriction operators $u^*: \POp(l)\rightarrow\POp(k)$,
associated to the injective maps $u: \{1<\dots<k\}\rightarrow\{1<\dots<l\}$,
and which reflect composites with the operation of arity zero $*\in\POp_+(0)$
in the operad $\POp_+$.
We call this structure a $\Lambda$-operad, where $\Lambda$ refers to the category which has the finite ordinals $\rset = \{1<\dots<r\}$
as objects and the injective maps between finite ordinals as morphisms.
Recall also that we use the phrase `unitary operads' for the category of operads such that $\POp(0) = *$.
Thus, the observation of~\cite{FresseBook} is that the category of unitary operads
is isomorphic to the category of $\Lambda$-operads.
In~\cite{FresseBook}, the phrase `non-unitary operads' is also used for the category of operads such that $\POp(0) = \varnothing$,
but we do not use this terminology in this paper, because we tacitely assume that our operads
are non-unitary in general.

We can dualize the notion of a $\Lambda$-operad to define a counterpart of this notion in the category of Hopf cooperads.
In~\cite{FresseBook}, we use Hopf $\Lambda$-cooperad structures with $\AOp(1) = \QQ$
in order to define a model for the rational homotopy of $\Lambda$-operads satisfying $\POp(0) = \varnothing$
and $\POp(1) = *$, and equivalently, for the rational homotopy of operads such that $\POp_+(0) = *$
and $\POp_+(1) = *$.
In this paper, we check, to complete our constructions, that we can also extend this model for the rational homotopy of unitary operads
to the case of operads with an arbitrary term in arity one. We address this subject in the fourth section of the paper.

Finally, we outline the applications of our constructions to the study of the framed little discs operads
in the concluding section of the paper.

\section{The model category of cochain dg-cooperads}\label{sec:modelcategories}
We explain the definition of our category of Hopf cochain dg-cooperads with full details in this section.
We also check that this category inherits a model structure.
In a preliminary step, we review the general definition of a cooperad.

\subsubsection{The general definition of a cooperad}\label{cooperads:general definition}
Briefly recall that a cooperad $\COp$ in a symmetric monoidal category $\CCat$ consists of a collection of objects $\COp(r)\in\CCat$, $r>0$,
together with an action of the symmetric group on $r$ letters on $\COp(r)$ for each $r>0$,
a counit morphism $\eta^*: \COp(1)\rightarrow\unit$,
where $\unit$ denotes the unit object of our symmetric monoidal category $\CCat$,
and composition coproducts $\circ_i^*: \COp(k+l-1)\rightarrow\COp(k)\otimes\COp(l)$,
defined for all $k,l>0$, $i = 1,\dots,k$,
and which satisfy natural equivariance, counit, and coassociativity relations. These axioms imply that the composition coproducts of a cooperad
give rise to tree-wise composition operations $\rho_{\ttree}: \COp(r)\rightarrow\FreeOp^c_{\ttree}(\COp)$,
for each directed tree $\ttree$ with $r$-ingoing leaves, where $\FreeOp^c_{\ttree}(\COp)$
denotes the tensor product of cooperad components $\COp(r_v)$
associated to the vertices of our trees $v\in V(\ttree)$,
with an arity $r_v$ that correspond to the number of ingoing edges
of the vertex $v$ (see~\cite[Theorem II.9.1.9 and \S C.1]{FresseBook}).

In what follows, we only consider cooperads in the category of cochain graded dg-modules $\CCat = \dg^*\Mod$ (with $\kk = \QQ$ as ground field),
and in the category of commutative cochain dg-algebras $\CCat = \dg^*\ComCat$. We go back to the definition of these categories
later on, when we explain the definition of our model structure on cooperads.
In our constructions, we use the zero object of the category of cochain graded dg-modules, the direct sum operation of cochain graded dg-modules,
and the obvious lifting of the direct sum of cochain graded dg-modules
to the category of commutative cochain dg-algebras.
In our subsequent formulas, we also use the identity $\unit = \QQ$ for the unit object
of these categories $\CCat = \dg^*\Mod,\dg^*\ComCat$.

Let $\bar{\COp}$ denote the collection such that $\bar{\COp}(1) = \ker(\eta^*: \COp(1)\rightarrow\unit)$
and $\bar{\COp}(r) = \COp(r)$ for $r>0$.
In what follows, we assume that our cooperads are equipped with a coaugmentation $\epsilon^*: \unit\rightarrow\COp(1)$
which defines a section of the counit morphism of $\eta^*: \COp(1)\rightarrow\unit$
so that we have can naturally regard the collection $\bar{\COp}$
as a direct summand of the collection $\COp$
underlying our cooperad. We refer to this collection as the coaugmentation coideal of our cooperad.
Then we consider the reduced treewise coproducts $\bar{\rho}_{\ttree}: \bar{\COp}(r)\rightarrow\FreeOp^c_{\ttree}(\bar{\COp})$
given by the composite of the treewise coproducts associated to our cooperad on $\bar{\COp}(r)$
with the morphism $\COp(r_v)\rightarrow\bar{\COp}(r_v)$ induced by the canonical projection $\COp\rightarrow\bar{\COp}$
on each factor $\COp(r_v)$, $v\in V(\ttree)$, of the treewise tensor product $\FreeOp^c_{\ttree}(\COp)(r) = \bigotimes_{v\in V(\ttree)}\COp(r_v)$.
For our purpose, we assume in our definition of a cooperad
that the reduced treewise composition coproducts $\bar{\rho}_{\ttree}(\gamma)\in\FreeOp^c_{\ttree}(\bar{\COp})$
associated to a given element $\gamma\in\bar{\COp}(r)$ vanish for all but a finite number of trees $\ttree$.
We refer to this requirement as the conilpotence condition.
We use the notation $\Op_0^c = \CCat\Op_0^c$ for this category of conilpotent cooperads in $\CCat$,
where we assume that the morphisms $\phi: \COp\rightarrow\DOp$
preserve all structure morphisms attached our objects,
including the coaugmentations.

\subsubsection{The definition of cofree cooperads}
In our subsequent constructions, we notably use the conilpotence condition when we define cofree objects in $\CCat\Op_0^c$.
To be explicit, let $\Seq_{>0}^c = \CCat\Seq_{>0}^c$ denote the category formed by the collections $\NOp = \{\NOp(r),r>0\}$
such that $\NOp(r)\in\CCat$ is equipped with an action of the symmetric group on $r$ letters $\Sigma_r$,
for each $r>0$ (the category of symmetric sequences in the terminology of~\cite[\S II.8.1]{FresseBook}).
We set $\FreeOp^c(\NOp)(r) = \bigoplus_{[\ttree]}\FreeOp_{\ttree}^c(\NOp)$, for each $r>0$,
where the direct sum ranges over a set of representatives of isomorphism classes
of trees with $r$ ingoing edges.
This collection inherits a natural cooperad structure
with coproduct operations $\circ_i^*: \FreeOp^c(\NOp)(k+l-1)\rightarrow\FreeOp^c(\NOp)(k)\otimes\FreeOp^c(\NOp)(l)$
given termwise by the sum of the obvious isomorphisms $\FreeOp_{\thetatree}^c(\NOp)\xrightarrow{\simeq}\FreeOp_{\stree}^c(\NOp)\otimes\FreeOp_{\ttree}^c(\NOp)$
associated to the tree decompositions $\thetatree\simeq\stree\circ_i\ttree$,
where $\stree\circ_i\ttree$ denotes the tree obtained by plugging in the outgoing edge of the tree $\ttree$
into the $i$th ingoing edge of the tree $\stree$.

In comparison with the construction of~\cite[Theorem II.9.1.9 and \S C.1]{FresseBook},
we just allow vertices $v$ with a single ingoing edge $r_v = 1$ in our trees.
Note simply that the number of tree decompositions $\thetatree\simeq\stree\circ_i\ttree$
is still finite in this setting.
Let us mention that we assume by convention $r_v>0$ for each vertex $v$ in the definition of a tree $\ttree$,
because our symmetric sequences $\NOp = \{\NOp(r),r>0\}$
are only defined in arity $r>0$.
The counit morphism associated to the cofree cooperad $\eta^*: \FreeOp^c(\NOp)(1)\rightarrow\unit$
is given by the projection of our sum onto the summand $\FreeOp_{\unittree}^c(\NOp)$
associated to the unit tree $\ttree = \unittree$, which has no vertex and consists of a single edge,
opened at all extremities.
The coaugmentation $\epsilon^*: \unit\rightarrow\FreeOp^c(\NOp)(1)$
is given by the canonical embedding of this summand $\FreeOp_{\unittree}^c(\NOp)$
into $\FreeOp^c(\NOp)(1)$. We easily check that all but a finite number of reduced treewise coproducts
vanish on each summand $\FreeOp_{\ttree}^c(\NOp)$
of the cofree cooperad, and the definition of the cofree cooperad $\FreeOp^c(\NOp)$ as a direct sum of these summands implies
that this object does satisfy the conilpotence condition
of our definition of a cooperad.

We easily check that the functor $\FreeOp^c: \NOp\rightarrow\FreeOp^c(\NOp)$
defines a right adjoint of the coaugmentation coideal functor $\overline{\omega}: \COp\mapsto\bar{\COp}$
from the category of cooperads $\Op_0^c$
to the category of symmetric sequences $\Seq_{>0}^c$.
The unit morphism of this adjunction $\rho: \COp\rightarrow\FreeOp^c(\bar{\COp})$
is the morphism given by the reduced treewise composition coproducts
of our cooperad $\bar{\rho}_{\ttree}: \bar{\COp}(r)\rightarrow\FreeOp^c_{\ttree}(\bar{\COp})$
on each term of our sum such that $\ttree\not=\unittree$
and by the counit morphism $\eta^*: \COp(1)\rightarrow\unit$
on the summand associated to the unit tree $\ttree = \unittree$.

\subsubsection{The model category of cooperads in cochain graded dg-modules}\label{cooperads:model structure}
The category of cochain graded dg-modules $\dg^*\Mod$ consists of the non-negatively upper graded modules $K = \oplus_{n\in\NN} K^n$
equipped with a differential $\delta: K\rightarrow K$
which raises degrees by one. We equip this category $\dg^*\Mod$
with the usual tensor product of differential graded modules,
with a symmetry isomorphism
which follows the usual sign rule of differential graded algebra.
For simplicity, we assume that we take $\kk = \QQ$ as a ground ring in all our definitions,
and that our modules are defined over this ground field.
(We are just consider the obvious extension of our constructions to the case $\kk = \RR$ in the concluding section of the paper.)

In what follows, we generally use the prefix `dg' for any category of objects defined in a category of differential graded modules,
whereas the phrase `cochain graded' refers to the range
of grading in the definition our dg-objects. In the case of commutative algebras and cooperads,
we just keep the adjective `cochain'
in our terminology. In particular, we use the expression cochain dg-cooperads
for the category of operads in $\dg^*\Mod$. For short, we also
adopt the notation $\dg^*\Op_0^c = \dg^*\Mod\Op_0^c$
for this category of cooperads.

The category of cochain graded dg-modules is equipped with a model structure such that the weak-equivalences are the morphisms $f: K\xrightarrow{\sim}L$
which induce an isomorphism in homology (the quasi-isomorphisms in the standard terminology of differential graded algebra),
the cofibrations are the morphisms $f: K\rightarrowtail L$ which are injective in positive degrees,
and the fibrations are the morphisms $f: K\twoheadrightarrow L$
which are surjective in all degrees.
We transport this model structure to the category of cochain dg-cooperads.
We explicitly assume that a morphism of cochain dg-cooperads is:
\begin{enumerate}
\item a weak-equivalence $\phi: \COp\xrightarrow{\sim}\DOp$
if this morphism defines a weak-equivalence of cochain graded dg-modules $\phi: \COp(r)\xrightarrow{\sim}\DOp(r)$
in each arity $r>0$ (thus, if this morphism induces an isomorphism in cohomology),
\item a cofibration $\phi: \COp\rightarrowtail\DOp$
if this morphism defines a cofibration of cochain graded dg-modules $\phi: \COp(r)\rightarrowtail\DOp(r)$
in each arity $r>0$ (thus, if this morphism is injective in positive degrees arity-wise),
\item and a fibration $\phi: \COp\twoheadrightarrow\DOp$
if this morphism has the right lifting property with respect to the class of acyclic cofibrations
determined by the above definitions.
\end{enumerate}
We then have the following statement:

\begin{thm}\label{thm:cooperad model category}
The definitions of the previous paragraph provide the category of cochain dg-cooperads $\dg^*\Op_0^c$
with a valid model structure.
Furthermore, this category is cofibrantly generated with, as a set of generating (acyclic) cofibrations,
the (acyclic) cofibrations
of cochain dg-cooperads $i: \COp\rightarrow\DOp$
such that the components $\COp(r)$ and $\DOp(r)$ of our cooperads $\COp$ and $\DOp$ form a cochain graded dg-module of countable dimension
in arity $r=1$, form a bounded cochain graded dg-module of finite dimension over the ground field in arity $r>1$,
and vanish for $r\gg 0$.
\end{thm}

\begin{proof}
We use almost the same plan and the same arguments as in the proof of~\cite[Theorem II.9.2.2]{FresseBook}.
We just need to adapt the proof that the (acyclic) cofibrations are given by filtered colimits of generating (acyclic) cofibrations
of the form considered in our statement.

To be explicit, in order to establish such a result in~\cite[\S II.9.2]{FresseBook},
we mainly prove that, for any injective morphism of cochain graded dg-cooperads $i: \COp\hookrightarrow\DOp$,
and for any collection $\SOp\subset\DOp$ such that $\SOp(r)$ vanishes in arity $r>m$
and forms a bounded cochain graded dg-module of finite dimension over the ground field for $r = 1,\dots,m$,
for some arity bound $m>0$, we can pick a subcooperad $\KOp\subset\DOp$ such that $\SOp\subset\KOp$
and which satisfies the same finiteness properties as $\SOp$.
Then we consider the injective morphism $i\shortmid_{\KOp\cap\COp}: \KOp\cap\COp\hookrightarrow\KOp$
which defines a cofibration of the form considered in our theorem.
If $i: \COp\hookrightarrow\DOp$ is also a weak-equivalence, then we can pick such a subcooperad $\KOp\subset\DOp$
so that the morphism $i\shortmid_{\KOp\cap\COp}: \KOp\cap\COp\hookrightarrow\KOp$
is a weak-equivalence as well (see~\cite[Lemma II.9.2.5]{FresseBook}).
The construction of~\cite[Lemma II.9.2.5]{FresseBook} only produces the components of arity $r>1$ of this subcooperad $\KOp\subset\DOp$.
In order to complete this construction,  we can adapt the arguments of~\cite[Proposition 1.5]{GetzlerGoerss}
to get a dg-coalgebra $\KOp(1)\subset\DOp(1)$, of countable dimension over the ground field,
and which contains $\SOp(1)\subset\DOp(1)$ as well as the factors of arity one
of the composition coproducts of the elements of the modules $\KOp(r)$, $r>1$,
in the cooperad $\DOp$. We then get that the whole collection $\KOp(r)$, $r>0$, forms a subobject of~$\DOp$
in the category of cochain dg-cooperads.
In the case where $i: \COp\hookrightarrow\DOp$ is a weak-equivalence, we use the construction of~\cite[Lemma 2.5]{GetzlerGoerss}
to produce a coalgebra $\KOp(1)\subset\DOp(1)$
with the additional property that the morphism $i\shortmid_{\KOp(1)\cap\COp(1)}: \KOp(1)\cap\COp(1)\hookrightarrow\KOp(1)$
is also a weak-equivalence. Then we still get that our extended cooperad morphism $i\shortmid_{\KOp\cap\COp}: \KOp\cap\COp\hookrightarrow\KOp$
defines a weak-equivalence in every arity $r>0$.
\end{proof}

\subsubsection{The model category of Hopf cochain dg-cooperads}\label{Hopf cooperads:model structure}
Recall that we call Hopf cochain dg-cooperad the structure defined by a cooperad in the category of commutative cochain dg-algebras,
where a commutative cochain dg-algebra consists of a commutative algebra
in the category of cochain graded dg-modules.
Note that in the case of a Hopf cooperad $\AOp$, the section $\epsilon^*: \QQ\rightarrow\AOp(1)$
of the cooperad counit $\eta^*: \AOp(1)\rightarrow\QQ$
is necessarily identified with the unit morphism of the dg-algebra $\AOp(1)$,
and the coaugmentation coideal of our object $\bar{\AOp}$
is obviously identified with the augmentation ideal of this dg-algebra $\AOp(1)$
in arity one.

In what follows, we use the short notation $\dg^*\Hopf\Op_0^c = \dg^*\ComCat\Op_0^c$
for the category of Hopf cochain dg-cooperads.
We also use the notation $\dg^*\Hopf\Seq_{>0}^c = \dg^*\ComCat\Seq_{>0}^c$ for the category of Hopf symmetric sequences,
which consists of the symmetric sequences in the category of commutative cochain dg-algebras.
We observed in~\cite[Proposition II.9.3.4]{FresseBook} that the cofree cooperad functor from the category of Hopf symmetric sequences
to the category of Hopf cooperads defines a lifting of the corresponding cofree cooperad functor from the category of symmetric sequences
to the category of cooperads (with no Hopf structure).
We readily check that this result remains valid in the setting of this paper.

We also have an obvious forgetful functor $\omega: \dg^*\Hopf\Op_0^c\rightarrow\dg^*\Op_0^c$
from the category of Hopf cochain dg-cooperads to the category of cochain dg-cooperads.
We observed in~\cite[\S II.9.3]{FresseBook} that this functor admits a left adjoint,
which is given by a cooperadic realization of the symmetric algebra $\Sym_{\Op_0^c}(-)$.
We easily check that this result remains valid in the setting of this paper too, but we have to adapt the explicit construction in the general setting.
To be explicit, for a cooperad $\COp\in\dg^*\Op_0^c$, we actually have the identity $\Sym_{\Op_0^c}(\COp)(1) = \Sym(\COp(1))\otimes_{\Sym(\QQ)}\QQ$
in arity one, where we form a relative tensor product to identify the morphism $\Sym(\epsilon^*): \Sym(\QQ)\rightarrow\Sym(\COp(1))$
induced by the coaugmentation morphism of our cooperad $\epsilon^*: \QQ\rightarrow\COp(1)$
with the unit morphism of the symmetric algebra $\eta: \QQ\rightarrow\Sym(\COp(1))$.
For $r>1$, we just have $\Sym_{\Op_0^c}(\COp)(r) = \Sym(\COp(r))$.

We use the adjunction $\Sym_{\Op_0^c}(-): \dg^*\Op_0^c\rightleftarrows\dg^*\Hopf\Op_0^c :\omega$
to transport the model structure on the category of cochain dg-cooperads
to the category of Hopf cochain dg-cooperads.
We explicitly assume that a morphism of Hopf cochain dg-cooperads is:
\begin{enumerate}
\item a weak-equivalence $\phi: \AOp\xrightarrow{\sim}\BOp$ if this morphism defines a weak-equivalence
in the category of cochain dg-coperads  (thus, if this morphism induces an isomorphism in cohomology),
\item a fibration $\phi: \AOp\twoheadrightarrow\BOp$ if this morphism defines a fibration in the category of cochain dg-coperads,
\item a cofibration $\phi: \AOp\rightarrowtail\BOp$ if this morphism has the left lifting property with respect to the class of acyclic fibrations
determined by the above definitions.
\end{enumerate}
We then have the following statement:

\begin{thm}\label{thm:Hopf cooperad model category}
The definitions of the previous paragraph provide the category of Hopf cochain dg-cooperads $\dg^*\Hopf\Op_0^c$ with a valid model structure.
Furthermore, this category is cofibrantly generated with the morphisms $\Sym_{\Op_0^c}(i): \Sym_{\Op_0^c}(\COp)\rightarrow\Sym_{\Op_0^c}(\DOp)$
defined by taking the image of the generating (acyclic) cofibrations $i: \COp\rightarrow\DOp$
of the model category of cooperads under the symmetric algebra functor $\Sym_{\Op_0^c}(-): \dg^*\Op_0^c\rightarrow\dg^*\Hopf\Op_0^c$
as a set of generating (acyclic) cofibrations.
\end{thm}

\begin{proof}
We use exactly the same verifications as in the proof of the analogous statement in~\cite[Theorem II.9.3.9]{FresseBook}.
\end{proof}

We also have the following additional result which follows from the description of the generating (acyclic) cofibrations of Hopf cochain dg-cooperads
given in this theorem:

\begin{prop}\label{prop:Hopf cooperad cofibrations}
The cofibrations of Hopf cochain dg-cooperads $\phi: \AOp\rightarrowtail\BOp$
define a cofibration of commutative cochain dg-algebras $\phi: \AOp(r)\rightarrowtail\BOp(r)$
in each arity $r>0$ (and this cofibration is also acyclic if $\phi$ is so).
\end{prop}

\begin{proof}
We immediately see that the generating (acyclic) cofibrations of our category of Hopf cochain dg-cooperads
$\Sym_{\Op_0^c}(i): \Sym_{\Op_0^c}(\COp)\rightarrow\Sym_{\Op_0^c}(\DOp)$
define an (acyclic) cofibration of commutative cochain dg-algebra
in each arity. Then we just use that the colimits of Hopf cochain dg-cooperads are created arity-wise
in the underlying category of commutative cochain dg-algebras
to deduce from this result that every relative cell complex of generating (acyclic) cofibrations of Hopf cochain dg-cooperads,
and hence, every (acyclic) cofibration,
defines an (acyclic) cofibration in the category of commutative cochain dg-algebras,
which is the claim of our proposition.
\end{proof}

In our verifications, we also use an explicit description, in terms of conormalized cochain complexes,
of the totalization of cosimplicial objects in the category of cochain dg-cooperads.
This description follows from a straightforward generalization of a result given in \cite[\S II.9.4]{FresseBook}
in the category of (Hopf) cochain dg-cooperads which reduce
to the ground field in arity one.

To be more explicit, to carry out our construction, we use the operadic cobar-bar adjunction between the category of dg-cooperads
and the category of dg-operads. We use the notation $\DGB^c(\COp)$ for the bar construction of a dg-cooperad $\COp$
and the notation $\DGB(\POp)$ for the bar construction of a dg-operad $\POp$.
We refer to~\cite[\S C.2]{FresseBook} for a detailed survey on this operadic application of the bar duality (see also \cite{GetzlerJones}
for the original reference on the formalism which we use in this paper).
We just allow arbitrary terms in arity one, whereas we assume that the components of our operads and cooperads
are reduced to the ground ring in~\cite{FresseBook} (as usual), but the extension of the constructions
to our setting is straightforward. Let us simply mention that the bar construction is only defined for dg-operads $\POp$
which are equipped with an augmentation over the operad $\IOp$
such that $\IOp(1) = \QQ$ and $\IOp(r) = 0$ for $r>0$. (This structure is dual
to the coaugmentation which we consider in the case of cooperads.)
For such an operad $\POp$, we use the notation $\bar{\POp}(1)$ for the kernel of the augmentation morphism $\epsilon: \POp(1)\rightarrow\QQ$.
The cobar construction of a cooperad $\POp = \DGB^c(\COp)$
is naturally equipped with such an augmentation. Moreover, when $\COp$ is a cooperad in the category of cochain graded dg-modules,
then the image of $\COp$ under the cobar-bar construction is still a a cooperad in cochain graded dg-modules $\DGB\DGB^c(\COp)\in\dg^*\Op_0^c$,
and the unit morphism of the cobar-bar adjunction defines a weak-equivalence $\COp\xrightarrow{\sim}\DGB\DGB^c(\COp)$
in this category $\dg^*\Op_0^c$. The cobar-bar construction $\DGB\DGB^c(\COp)$ also forms a fibrant object
in $\dg^*\Op_0^c$.
The proof of this claim follows from an easy generalization of the arguments given in~\cite[\S II.9.4.1]{FresseBook}
in the context of cooperads which reduce to the unit object in arity one. (We just need
to consider a filtration by the grading of the bar construction $\DGB(\COp)$,
rather than the filtration by the arity in the proof of the general proposition~\cite[Proposition II.9.2.10]{FresseBook}
which we use to establish our result in this reference.)

Let now $\POp = \DGB^c(\COp)$. We consider the simplicial dg-operad $\POp\bar{\otimes}\DGOmega^*(\Delta^{\bullet})$
such that $(\POp\bar{\otimes}\DGOmega^*(\Delta^{\bullet}))(1) = \QQ\oplus\bar{\POp}(1)\otimes\DGOmega^*(\Delta^{\bullet})$
and $(\POp\bar{\otimes}\DGOmega^*(\Delta^{\bullet}))(r) = \POp(r)\otimes\DGOmega^*(\Delta^{\bullet})$
for $r>0$,
where $\DGOmega^*(\Delta^{\bullet})$ denotes the Sullivan commutative cochain dg-algebra
of the collection of the simplices $\Delta^{\bullet}$.
Then one can check that $\DGB(\POp)^{\Delta^{\bullet}} = \DGB(\POp\bar{\otimes}\DGOmega^*(\Delta^{\bullet})$
defines a simplicial framing of the object $\DGB(\POp)$
in the category of cochain dg-cooperads (adapt the arguments given in~\cite[\S II.9.4]{FresseBook} again).
We precisely use this simplicial framing construction in order to get an explicit realization of the totalization of cosimplicial objects
in the category of cochain dg-cooperads. Then we get the following statement:

\begin{thm}\label{thm:cooperad totalization}
Let $\COp^{\bullet}\in\simp\dg^*\Op_0^c$ be a cosimplicial object of the category of cochain dg-cooperads.
Let $\POp^{\bullet} = \DGB^c(\COp^{\bullet})$ denote the image of this cosimplicial cochain dg-cooperad
under the operadic cobar construction $\DGB^c(-)$.
\begin{enumerate}
\item
The object $\DGB(\POp^{\bullet})^{\Delta^{\bullet}} = \DGB(\POp^{\bullet}\bar{\otimes}\DGOmega^*(\Delta^{\bullet})$,
where $\DGB(-)$ denotes the operadic bar construction,
defines a simplicial framing of the object $\KOp^{\bullet} = \DGB(\POp^{\bullet})$
in the category of cosimplicial cochain dg-cooperads,
so that we have the identity:
\begin{equation*}
\Tot(\DGB(\POp^{\bullet})) = \int_{\underline{n}\in\Delta}\DGB(\POp^n\bar{\otimes}\DGOmega^*(\Delta^n))
 = \DGB(\int_{\underline{n}\in\Delta}\POp^n\bar{\otimes}\DGOmega^*(\Delta^n))
\end{equation*}
in $\dg^*\Op_0^c$, where $\Tot(\KOp^{\bullet})$ denotes the totalization of the cosimplicial object $\KOp^{\bullet} = \DGB(\POp^{\bullet})$
in the sense of the theory of model categories.
\item
For such an object $\KOp^{\bullet} = \DGB\DGB^c(\COp^{\bullet})$, we have a chain of natural weak-equivalences of cochain dg-cooperads
\begin{equation*}
\Tot(\DGB\DGB^c(\COp^{\bullet})) = \DGB(\int_{\underline{n}\in\Delta}\POp^n\bar{\otimes}\DGOmega^*(\Delta^n))
\xrightarrow{\sim}\DGB\DGB^c(\DGN^*(\COp^{\bullet}))\xleftarrow{\sim}\DGN^*(\COp^{\bullet}),
\end{equation*}
where $\DGN^*(\COp^{\bullet})$ denotes the obvious arity-wise extension, to the category of cochain dg-cooperads,
of the conormalization complex functor on the category of cosimplicial cochain graded dg-modules $\DGN^*(-)$
(see~\cite[\S II.9.4.5]{FresseBook}).
For a constant cosimplicial object $\COp^{\bullet} = \COp$,
we have $\Tot(\DGB\DGB^c(\COp^{\bullet})) = \DGB\DGB^c(\COp)$, $\DGN^*(\COp^{\bullet}) = \COp$,
and the first weak-equivalence of this chain reduces to the identity morphism.
\end{enumerate}
\end{thm}

\begin{proof}
We use a straightforward generalization of the constructions and arguments given in the proof of the analogue of these statements
in~\cite[\S II.9.4]{FresseBook}.
\end{proof}

We also have the following statement:

\begin{prop}\label{prop:Hopf cooperad totalization}
The forgetful functor $\omega: \dg^*\Hopf\Op_0^c\rightarrow\dg^*\Op_0^c$
preserves the totalization of cosimplicial objects.
\end{prop}

\begin{proof}
This proposition follows from the observation that this forgetful functor creates limits
and from the definition of the class of weak-equivalences
and fibrations
in the model category of Hopf cochain dg-cooperads (see the proof of the analogous statement
in~\cite[Proposition II.9.4.13]{FresseBook}).
\end{proof}

\section{The operadic upgrading of the Sullivan model functor}\label{sec:operadic Sullivan model}
We now explain the definition of our adjunction $\DGG: \dg^*\Hopf\Op_0^c\rightleftarrows\Simp\Op_{\varnothing}^{op} :\DGOmega^*_{\sharp}$
between the category of Hopf cochain dg-cooperads $\dg^*\Hopf\Op_0^c$ and the category of operads
in simplicial sets $\Simp\Op_{\varnothing}$. To be more precise, we explain the definition of these functors in the next paragraph
and we prove afterwards that we have a weak-equivalence $\DGOmega^*_{\sharp}(\POp)(r)\sim\DGOmega^*(\POp(r))$
in each arity $r>0$ (under mild connectedness and finiteness requirements)
when we take the image of a cofibrant operad $\POp$
under our operadic upgrading of the Sullivan functor $\DGOmega^*_{\sharp}(-)$.

\subsubsection{The operadic enhancement of the Sullivan cochain dg-algebra functor}\label{operadic Sullivan model:construction}
Recall that the functor $\DGG: \dg^*\ComCat\rightarrow\Simp^{op}$ is defined by $\DGG(A) = \Mor_{\dg^*\ComCat}(A,\DGOmega^*(\Delta^{\bullet}))$
for any commutative cochain dg-algebra $A\in\dg^*\ComCat$.
We have an obvious isomorphism $\DGG(A\otimes B)\simeq\DGG(A)\otimes\DGG(B)$, for all $A,B\in\dg^*\ComCat$.
This result implies that the collection $\DGG(\AOp)(r) = \Mor_{\dg^*\ComCat}(\AOp(r),\DGOmega^*(\Delta^{\bullet}))$
associated to a Hopf cochain dg-cooperad $\AOp\in\dg^*\Hopf\Op_0^c$
inherits a natural operad structure so that the mapping $\DGG: \AOp\mapsto\DGG(\AOp)$
defines a functor $\DGG: \dg^*\Hopf\Op_0^c\rightarrow\Simp\Op_{\varnothing}^{op}$
from the category of Hopf cochain dg-cooperads $\dg^*\Hopf\Op_0^c$
to the category of operads in simplicial sets $\Simp\Op_{\varnothing}$.

We claim that this functor $\DGG: \dg^*\Hopf\Op_0^c\rightarrow\Simp\Op_{\varnothing}^{op}$
admits a right adjoint $\DGOmega^*_{\sharp}: \Simp\Op_{\varnothing}^{op}\rightarrow\dg^*\Hopf\Op_0^c$,
of which we deduce the existence from the adjoint functor lifting theorem.
In short, we first set $\DGOmega^*_{\sharp}(\FreeOp(\MOp)) = \FreeOp^c(\DGOmega^*(\MOp))$ for a free operad $\POp = \FreeOp(\MOp)$
generated by a symmetric sequence in simplicial sets $\MOp\in\Simp\Seq_{>0}$,
where $\DGOmega^*(\MOp)\in\dg^*\Hopf\Seq_{>0}^c$
denotes the Hopf symmetric sequence such that $\DGOmega^*(\MOp)(r) = \DGOmega^*(\MOp(r))$, for each $r>0$.
Then we use the Yoneda Lemma to extend this mapping $\DGOmega^*_{\sharp}: \FreeOp(\MOp)\mapsto\DGOmega^*_{\sharp}(\FreeOp(\MOp))$
to a functor from the full subcategory of operads generated by free objects
towards the category of Hopf cochain dg-cooperad.
Finally, we use that any operad admits a natural presentation as a reflexive coequalizer of free operads,
which our functor carries to free equalizers,
in order to extend our construction to the whole category of operads $\Simp\Op_{\varnothing}$.

We equip the category of operads in simplicial sets $\Simp\Op_{\varnothing}$ with its standard model structure,
where the weak-equivalences (respectively, the fibrations) are the morphisms of operads
which define a weak-equivalence (respectively, a fibration) of simplicial sets
in each arity, whereas the cofibrations are characterized by the left lifting property with respect to the class of acyclic fibrations
(see~\cite{BergerMoerdijk} and \cite[\S II.8.2]{FresseBook} for a detailed account
of the definition of this model structure).
Then we easily check that:

\begin{prop}\label{prop:operadic Sullivan model adjunction}
The functors $\DGG: \dg^*\Hopf\Op_0^c\rightleftarrows\Simp\Op_{\varnothing}^{op} :\DGOmega^*_{\sharp}$ define a Quillen adjunction.
\end{prop}

\begin{proof}
Recall that the (acyclic) cofibrations of Hopf cochain dg-cooperads form (acyclic) cofibrations of commutative cochain dg-algebras
arity-wise (see Proposition~\ref{prop:Hopf cooperad cofibrations}).
The left adjoint of the Sullivan cochain dg-algebra functor $\DGG$ carries such (acyclic) cofibrations
to (acyclic) fibrations in the category of simplicial sets
since this functor fits in Quillen adjunction
by the classical rational homotopy theory of spaces (see for instance~\cite[Proposition II.7.2.2]{FresseBook}
for this model category formulation of the observations of Sullivan).
We just use that the (acyclic) fibrations of the category of operads in simplicial sets are precisely the operad morphisms
which form an (acyclic) fibration in the category of simplicial sets
to deduce from this result that our functor $\DGG: \dg^*\Hopf\Op_0^c\rightarrow\Simp\Op_{\varnothing}^{op}$
carries the (acyclic) cofibrations of Hopf cochain dg-cooperads
to (acyclic) fibrations in the category of operads, and this proves the claim of this proposition.
\end{proof}

We have natural morphisms of commutative cochain dg-algebras $\chi: \DGOmega^*_{\sharp}(\POp)(r)\rightarrow\DGOmega^*(\POp(r))$
whose adjoints $\chi_{\sharp}: \DGG(\DGOmega^*_{\sharp}(\POp)(r))\rightarrow\POp(r)$ (when we take the adjunction relation
associated to the standard Sullivan functor on the category of simplicial sets)
define an operad morphism $\chi_{\sharp}: \DGG(\DGOmega^*_{\sharp}(\POp))\rightarrow\POp$,
which represent the augmentation morphism of our adjunction $\DGG: \dg^*\Hopf\Op_0^c\rightleftarrows\Simp\Op_{\varnothing}^{op} :\DGOmega^*_{\sharp}$
between Hopf cochain dg-cooperads and operads in simplicial sets.
We consider this comparison morphism $\chi: \DGOmega^*_{\sharp}(\POp)(r)\rightarrow\DGOmega^*(\POp(r))$, defined for any given arity $r>0$.
We aim to establish the following theorem:

\begin{thm}\label{thm:operadic Sullivan model validity}
Let $\POp$ be a cofibrant object of the category of operads in simplicial sets $\Simp\Op_{\varnothing}$. We assume that the space $\POp(1)$ is connected
and that the spaces $\POp(r)$ have a homology with rational coefficients $\DGH_*(\POp(r)) = \DGH_*(\POp(r),\QQ)$
which form a module of finite dimension over $\QQ$ in each degree. Then our comparison morphism defines a weak-equivalence
of commutative cochain dg-algebras
\begin{equation*}
\chi: \DGOmega^*_{\sharp}(\POp)(r)\xrightarrow{\sim}\DGOmega^*(\POp(r))
\end{equation*}
for each arity $r>0$.
\end{thm}

We devote the rest of this section to the proof of this result.
We face some difficulties when we consider free operads as in our construction of the functor $\DGOmega^*_{\sharp}: \POp\mapsto\DGOmega^*_{\sharp}(\POp)$,
because the functor $\DGOmega^*(-)$ carries the coproduct underlying the free operad to a product,
and this product is countable in each arity when we allow non-trivial terms in arity one.
We can not therefore go through free operad structures in our verification of the claim of the theorem.
We use an operadic analogue of the James construction to get round this difficulty.

\subsubsection{The operadic James construction}\label{operadic James construction}
We consider the category of under-objects $\IOp/\Simp\Seq_{>0}$ in the category of symmetric sequences in simplicial sets $\Simp\Seq_{>0}$
where $\IOp$ denotes the symmetric sequence such that $\IOp(1) = *$
and $\IOp(r) = \varnothing$ for $r>1$. This symmetric sequence $\IOp$ actually represents the initial object
in the category of operads.
We define the operadic James construction $\FreeOp_*(\MOp)$ associated to an object of this category $\MOp\in\IOp/\Simp\Seq_{>0}$
as the result of the following pushout in the category of operads:
\begin{equation*}
\xymatrix{ \FreeOp(\IOp)\ar[r]\ar[d] & \IOp\ar@{.>}[d] \\
\FreeOp(\MOp)\ar@{.>}[r] & \FreeOp_*(\MOp) },
\end{equation*}
where we consider the free operad morphism $\FreeOp(\IOp)\rightarrow\FreeOp(\MOp)$ associated to the canonical morphism $\IOp\rightarrow\MOp$
attached to our object $\MOp$ on the one hand, and the operad morphism $\FreeOp(\IOp)\rightarrow\IOp$
induced by the identity morphism of the symmetric sequence $\IOp$
on the other hand.

Intuitively, the morphism $\IOp\rightarrow\MOp$ is determined by the choice of a base point
in the component of arity one of our symmetric sequence $*\in\MOp(1)$,
and the above pushout makes the factors $*\in\MOp(1)$ equal to operadic units in the operad $\FreeOp_*(\MOp)$.


\subsubsection{The reduced cotriple resolution of operads}\label{reduced cotriple resolution}
We easily see that the mapping $\FreeOp_*: \MOp\mapsto\FreeOp_*(\MOp)$
defines a left adjoint of the obvious functor $\omega_*: \Simp\Op_{\varnothing}\rightarrow\IOp/\Simp\Seq_{>0}$
which forgets about the composition products in the definition of an operad $\POp\in\Simp\Op_{\varnothing}$,
but retains the operadic unit $1\in\POp(1)$.
We consider the cotriple resolution associated to this adjunction $\FreeOp_*: \IOp/\Simp\Seq_{>0}\rightleftarrows\Simp\Op_{\varnothing} :\omega_*$.
In what follows, we refer to this version of the cotriple resolution as the reduced cotriple resolution
in order to distinguish this object from the standard cotriple resolution of operads,
which we associate to the free operad adjunction $\FreeOp_*: \Simp\Seq_{>0}\rightleftarrows\Simp\Op_{\varnothing} :\omega$.
We usually drop the forgetful functor $\omega_*$ from our notation for short,
and we also write $\FreeOp_* = \omega_*\FreeOp_*$ for the composite of the operadic James construction $\FreeOp_*$
with the forgetful functor from operads to our category of coaugmented
symmetric sequences $\IOp/\Simp\Seq_{>0}$.
Then we define our reduced version of the cotriple resolution of an operad $\POp\in\Simp\Op_{\varnothing}$
as the simplicial object such that $\Res_*(\POp)_n = \FreeOp_*^{n+1}(\POp)$,
for each dimension $n\in\NN$.
The face operators $d_i: \Res_*(\POp)_n\rightarrow\Res_*(\POp)_{n-1}$
are defined by applying an adjunction augmentation $\lambda: \FreeOp_*\omega_*\rightarrow\Id$ on the $i+1$st factor $\FreeOp_* = \FreeOp_*\omega_*$
of the composite $\Res_*(\POp)_n = \FreeOp_*^{n+1}(\POp)$,
whereas the degeneracies $s_j: \Res_*(\POp)_n\rightarrow\Res_*(\POp)_{n+1}$
are defined by inserting an adjunction unit $\iota: \Id\rightarrow\omega_*\FreeOp_*$
inside the $j$th factor $\FreeOp_* = \FreeOp_*\omega_*$.
This simplicial object is equipped with an augmentation $\epsilon: \Res_*(\POp)_0\rightarrow\POp$,
which is also given by the augmentation morphism of our adjunction $\lambda: \FreeOp_*\omega_*\rightarrow\Id$
on the operad $\Res_*(\POp)_0 = \FreeOp_*(\POp)$. We have the following result:

\begin{prop}\label{prop:reduced cotriple resolution homotopy}
If $\POp$ is cofibrant as a symmetric sequence, then the reduced cotriple resolution $\Res_*(\POp)_{\bullet} = \FreeOp_*^{{\bullet}+1}(\POp)$
forms a Reedy cofibrant simplicial object in the category of operads in simplicial sets.
Furthermore, the augmentation $\Res_*(\POp)_{\bullet}\rightarrow\POp$ induces a weak-equivalence $|\Res_*(\POp)_{\bullet}|\xrightarrow{\sim}\POp$
when we pass to the geometric realization in $\Simp\Op_{\varnothing}$.
\end{prop}

\begin{proof}
The proof of this proposition follows from a straightforward adaptation of the observations of~\cite[\S II.8.5]{FresseBook},
where we prove that the standard (unreduced) cotriple resolution forms a Reedy cofibrant simplicial object
in the category of operads (when our operad is cofibrant as a symmetric sequence)
and that the augmentation associated to this resolution also induces a weak-equivalence of operads
when we pass to the geometric realization.
Let us simply mention that the proof that the augmentation $\Res_*(\POp)_{\bullet}\rightarrow\POp$ induces a weak-equivalence $|\Res_*(\POp)_{\bullet}|\xrightarrow{\sim}\POp$ when we pass to the geometric realization follows from the observation
that the reduced cotriple resolution $\Res_*(\POp)_{\bullet}$
is endowed with extra degeneracies $s_{-1}: \Res_*(\POp)_n\rightarrow\Res_*(\POp)_{n+1}$
which are defined by inserting an adjunction unit $\iota: \Id\rightarrow\omega_*\FreeOp_*$
in front of our composite functor $\Res_*(\POp)_{\bullet} = \FreeOp_*^{{\bullet}+1}(\POp)$. (But these extra degeneracies
are only defined in the category of symmetric sequences.)
\end{proof}

We use the reduced cotriple resolution $\Res_*(\POp)_{\bullet}$ in order to reduce the proof of our main theorem
to the case of the operadic James construction.
We check the validity of our claim on this case first. We rely on the following observation:

\begin{prop}\label{prop:operadic James construction model}
For the operadic James construction $\POp = \FreeOp_*(\MOp)$, we have an identity:
\begin{equation*}
\DGOmega_{\sharp}^*(\FreeOp_*(\MOp)) = \FreeOp^c(\overline{\DGOmega^*(\MOp)})
\end{equation*}
in the category of cochain dg-cooperads, where $\overline{\DGOmega^*(\MOp)}$ denotes the symmetric collection
such that $\overline{\DGOmega^*(\MOp)}(1) = \ker(\DGOmega^*(\MOp)(1)\rightarrow\DGOmega^*(\pt))$,
and $\overline{\DGOmega^*(\MOp)}(r) = \DGOmega^*(\MOp)(r)$
for $r>1$.
\end{prop}

\begin{proof}
The functor $\DGOmega^*_{\sharp}$ carries the pushout diagram in the definition of the operadic James construction
to a pullback diagram of the form
\begin{equation*}
\xymatrix{ \DGOmega^*_{\sharp}(\FreeOp_*(\MOp))\ar@{.>}[d]\ar@{.>}[r] & \IOp^c\ar[d] \\
\FreeOp^c(\DGOmega^*(\MOp))\ar[r] & \FreeOp^c(\DGOmega^*(\IOp)) }
\end{equation*}
in the category of Hopf cochain dg-cooperads, where $\IOp^c$ denotes the cooperad such that $\IOp^c(1) = \QQ$ and $\IOp^c(r) = 0$ for $r>1$.
Note simply that we have $\DGOmega^*_{\sharp}(\IOp) = \IOp^c$.
In fact, we can identify the collection $\IOp^c$ with the final object in the category of Hopf cochain dg-cooperads,
and this relation $\DGOmega^*_{\sharp}(\IOp) = \IOp^c$ also follows from the observation that the right adjoint functor $\DGOmega^*_{\sharp}$
carries any colimit in the category of operads in simplicial sets
to a limit in the category of Hopf cochain dg-cooperads.

We trivially have $\DGOmega^*(\IOp) = \IOp^c$, and $\DGOmega^*(\MOp) = \overline{\DGOmega^*(\MOp)}\oplus\IOp^c$
when we forget about commutative cochain dg-algebra structures, with the splitting induced by the map $*\rightarrow\MOp(1)$
attached to our object $\MOp\in\IOp/\Simp\Seq_{>0}$.
We moreover have $\IOp^c = \FreeOp^c(0)$, where we consider the cofree cooperad on the null symmetric sequence $0$,
and the right-hand side vertical morphism of our diagram
is identified with the morphism of cofree cooperads
induced by the null morphism
of symmetric sequences $0\rightarrow\DGOmega^*(\IOp) = \IOp^c$.
Recall that the forgetful functor from Hopf cochain dg-cooperads to cochain dg-cooperads creates limits.
The result of the lemma just follows from the observation that the cofree cooperad functor carries the obvious pullback diagram
\begin{equation*}
\xymatrix{ \overline{\DGOmega^*(\MOp)}\ar@{.>}[d]\ar@{.>}[r] & 0\ar[d] \\
\DGOmega^*(\MOp)\ar[r] & \DGOmega^*(\IOp) }
\end{equation*}
in the category of symmetric sequences to a pullback diagram in the category of cochain dg-cooperads.
\end{proof}

We use the result of the previous proposition to establish the following statement:

\begin{prop}\label{prop:operadic James construction model comparison}
If $\MOp(1)$ is connected and the spaces $\MOp(r)$, $r>0$, have a rational homology $\DGH_*(\MOp(r)) = \DGH_*(\MOp(r),\QQ)$
that form a module of finite dimension over $\QQ$ degreewise,
then the comparison morphism associated to the operadic James construction $\POp = \FreeOp_*(\MOp)$
defines a weak-equivalence $\chi: \DGOmega^*_{\sharp}(\FreeOp_*(\MOp))(r)\xrightarrow{\sim}\DGOmega^*(\FreeOp_*(\MOp)(r))$.
\end{prop}

\begin{proof}
Recall that the cofree cooperad $\FreeOp^c(\NOp)$ associated to a symmetric collection $\NOp\in\dg^*\Seq_{>0}^c$
admits an expansion of the form $\FreeOp^c(\NOp)(r) = \bigoplus_{[\ttree]}\FreeOp^c_{\ttree}(\NOp)$,
for each arity $r>0$,
where we consider the treewise tensor products $\FreeOp^c_{\ttree}(\NOp) = \bigotimes_{v\in V(\ttree)}\NOp(r_v)$,
associated to a set of representatives of isomorphism classes of trees with $r$-ingoing edges $\ttree$.
For the free operad $\FreeOp(\MOp)$ associated to a symmetric collection in simplicial sets $\MOp\in\Simp\Seq_{>0}$,
we dually have $\FreeOp(\MOp) = \coprod_{[\ttree]}\FreeOp_{\ttree}(\MOp)$
with $\FreeOp_{\ttree}(\MOp) = \times_{v\in V(\ttree)}\MOp(r_v)$.
Each element of the operadic James construction $x\in\FreeOp_*(\MOp)$
has a reduced form $x^r = (x_v)_{v\in V(\ttree)}\in\FreeOp_{\ttree}(\MOp)$
such that $x_v\in\MOp(1)\setminus *$ for each $v\in V(\ttree)$
with $r_v = 1$.

The Sullivan cochain functor is defined by $\DGOmega^*(X) = \Mor_{\Simp}(X,\DGOmega^*(\Delta^{\bullet}))$,
for each simplicial set $X\in\Simp$. For a cochain $\alpha\in\DGOmega^*(X)$,
we use the notation $\alpha(x)\in\DGOmega^*(\Delta^n)$
for the form on the $n$-simplex $\Delta^n$
associated to an element $x\in X_n$.
Let $\alpha = \bigotimes_v\alpha_v\in\FreeOp^c_{\ttree}(\overline{\DGOmega^*(\MOp)})$
be an element of the Hopf cooperad $\DGOmega^*_{\sharp}(\FreeOp_*(\MOp))$
in the representation of Proposition~\ref{prop:operadic James construction model}.
We easily check that the image of such a tensor $\alpha = \bigotimes_v\alpha_v\in\FreeOp^c_{\ttree}(\overline{\DGOmega^*(\MOp)})$
under our comparison morphism $\chi: \DGOmega^*_{\sharp}(\FreeOp_*(\MOp))(r)\rightarrow\DGOmega^*(\FreeOp_*(\MOp)(r))$
is identified with the form $\chi(\bigotimes_v\alpha_v)\in\DGOmega^*(\FreeOp_*(\MOp)(r))$
such that $\chi(\bigotimes_v\alpha_v)(x) = \prod_v\alpha_v(x)$ if the element $x\in\FreeOp_*(\MOp)$
satisfies $x^r = (x_v)_{v\in V(\ttree)}\in\FreeOp_{\ttree}(\MOp)$
and $\chi(\bigotimes_v\alpha_v)(x) = 0$ otherwise.

Let $\FreeOp^{\leq m}_*(\MOp)(r)$ denote the subset of the James construction that consists of classes
of the treewise cartesian products $(x_v)_{v\in V(\ttree)}\in\FreeOp_{\ttree}(\MOp)$
such that $\sharp V(\ttree)\leq m$ in the free operad $\FreeOp(\MOp)$.
We have a cofiber sequence of simplicial sets
\begin{equation*}
\xymatrix{ \FreeOp^{\leq m-1}_*(\MOp)(r)\ar[r] & \FreeOp^{\leq m}_*(\MOp)(r)\ar[r] & \bigvee_{\sharp V(\ttree) = m}\FreeOp_{\ttree}^{\wedge}(\MOp) },
\end{equation*}
where $\FreeOp_{\ttree}^{\wedge}(\MOp)$ denotes the quotient of the tree-wise cartesian product $\FreeOp_{\ttree}(\MOp) = \times_v\MOp(r_v)$
over the subspace that consists of the collections $(x_v)_{v\in V(\ttree)}\in\FreeOp_{\ttree}(\MOp)$
such that $x_v = *$ for some $v\in V(\ttree)$
with $r_v = 1$.
Let $\FreeOp^c_{\leq m}(\overline{\DGOmega^*(\MOp)})\subset\FreeOp^c(\overline{\DGOmega^*(\MOp)})$
denote the symmetric sequence
such that $\FreeOp^c_{\leq m}(\overline{\DGOmega^*(\MOp)})(r) = \bigoplus_{\sharp V(\ttree)\leq m}\FreeOp^c_{\ttree}(\overline{\DGOmega^*(\MOp)})$.
We see that our comparison morphism arises as the limit
of a tower of comparison maps $\chi: \FreeOp^c_{\leq m}(\overline{\DGOmega^*(\MOp)})(r)\rightarrow\DGOmega^*(\FreeOp^{\leq m}_*(\MOp)(r))$
which fit in commutative diagrams of the form:
\begin{equation*}
\xymatrix{ \bigoplus_{\sharp V(\ttree) = m}\FreeOp^c_{\ttree}(\overline{\DGOmega^*(\MOp)})\ar[r]\ar@{.>}[d] &
\FreeOp^c_{\leq m}(\overline{\DGOmega^*(\MOp)})\ar[d]\ar[r] &
\FreeOp^c_{\leq m-1}(\overline{\DGOmega^*(\MOp)})\ar[d] \\
\overline{\DGOmega}^*(\bigvee_{\sharp V(\ttree) = m}\FreeOp_{\ttree}^{\wedge}(\MOp))\ar[r] &
\DGOmega^*(\FreeOp_*^{\leq m}(\MOp)(r))\ar[r] &
\DGOmega^*(\FreeOp_*^{\leq m-1}(\MOp)(r)) },
\end{equation*}
where we now use the notation $\overline{\DGOmega}^*(X)$ for the augmentation ideal of the Sullivan cochain dg-algebra
of the pointed simplicial set $X = \bigvee_{\sharp V(\ttree) = m}\FreeOp_{\ttree}^{\wedge}(\MOp)$.
We easily deduce from an appropriate version of the K\"unneth theorem that the left-hand side vertical morphism of this diagram
defines a weak-equivalence of cochain graded dg-modules.
We obtain by induction that the medium vertical map defines a weak-equivalence as well for each $m\geq 0$.

Then the assumption that the simplicial set $\MOp(1)$ is connected
implies that the map $\DGOmega^*(\FreeOp_*^{\leq m}(\MOp)(r))\rightarrow\DGOmega^*(\FreeOp_*^{\leq m-1}(\MOp)(r))$
induces a surjection in cohomology, for each fixed arity $r>0$,
when $m$ is large enough with respect to the degree. We conclude that our comparison map induces a weak-equivalence
on the limit of our tower, and this result finishes the proof of our proposition.
\end{proof}

\begin{proof}[Proof of Theorem~\ref{thm:operadic Sullivan model validity}]
We now use the same argument lines as in~\cite[\S II.10.1]{FresseBook} to complete the proof of our main theorem
in the general case of a cofibrant operad $\POp\in\Simp\Op_{\varnothing}$.
We just give a short overview of the plan of this proof. We consider the reduced cotriple resolution
of our operad $\Res_*(\POp)_{\bullet} = \FreeOp_*^{{\bullet}+1}(\POp)$.
The weak-equivalence $\epsilon: |\ROp_{\bullet}|\xrightarrow{\sim}\POp$ of Proposition~\ref{prop:reduced cotriple resolution homotopy}
induces a weak-equivalence of Hopf cochain dg-cooperads
\begin{equation*}
\DGOmega^*_{\sharp}(\POp)\xrightarrow{\sim}\DGOmega^*_{\sharp}(|\ROp_{\bullet}|),
\end{equation*}
because the result of Proposition~\ref{prop:reduced cotriple resolution homotopy} also implies that the operad $\ROp = |\ROp_{\bullet}|$ is cofibrant,
and the functor $\DGOmega^*_{\sharp}$ preserves weak-equivalences between cofibrant objects
by Quillen adjunction. We moreover have an identity
\begin{equation*}
\DGOmega^*_{\sharp}(|\ROp_{\bullet}|) = \int_{\underline{n}\in\Delta}\DGOmega^*_{\sharp}(\ROp_n\otimes\Delta^n)
\end{equation*}
by adjunction, where $\ROp_{\bullet}\otimes\Delta^{\bullet}$ denotes a cosimplicial framing of the simplicial object $\ROp_{\bullet}$
in the category of operads in simplicials sets. By Quillen adjunction, we can also identify the right-hand side
of this formula with the totalization of the cosimplicial object $\DGOmega^*_{\sharp}(\ROp_{\bullet})$
in the category of Hopf cochain dg-cooperads, and hence, in the category of cochain dg-cooperads
since the forgetful functor preserves the totalization by Proposition~\ref{prop:Hopf cooperad totalization}.

The comparison map of Theorem~\ref{thm:operadic Sullivan model validity} fits in a commutative square:
\begin{equation}\tag{$*$}\label{eq:comparison map resolution}
\xymatrix{ \DGOmega^*_{\sharp}(\POp)(r)\ar[r]^-{(1)}\ar@{.>}[d] & \DGN^*(\DGOmega^*_{\sharp}(\ROp_{\bullet})(r))\ar[d]^{(3)} \\
\DGOmega^*(\POp)(r)\ar[r]^-{(2)} & \DGN^*(\DGOmega^*(\ROp_{\bullet})(r)) },
\end{equation}
where the horizontal arrows (1-2) are induced by augmentation of the reduced cotriple resolution $\epsilon: \ROp_{\bullet}\rightarrow\POp$.
We easily check again, by using the extra-degeneracies of the reduced cotriple resolution
(as in the proof of Proposition~\ref{prop:reduced cotriple resolution homotopy}),
that the lower horizontal arrow of this diagram is a weak-equivalence.
We use the results of Theorem~\ref{thm:cooperad totalization} to establish that the upper horizontal arrow is a weak-equivalence too.
To be explicit, we check that our morphism (1) fits in the following commutative diagram of weak-equivalences
when we forget about Hopf structures and we work in the category of cochain dg-cooperads:
\begin{equation*}
\xymatrix{ \DGOmega_{\sharp}^*(\POp)\ar[r]^-{\sim}\ar[dd]_-{\sim} &
\DGOmega_{\sharp}^*(|\ROp_{\bullet}|)\ar[r]^-{\simeq} &
\int_{\underline{n}\in\Delta}\DGOmega_{\sharp}^*(\ROp_n\otimes\Delta^n)\ar[d]_{(4)}^-{\sim} \\
&& \cdot \\
\DGB\DGB^c(\DGOmega_{\sharp}^*(\POp))\ar@{.>}[rr] &&
\int_{\underline{n}\in\Delta}\DGB(\DGB^c(\DGOmega_{\sharp}^*(\ROp_n))\bar{\otimes}\DGOmega^*(\Delta^n))\ar[u]^{(5)}_-{\sim}\ar[d]_{(6)}^-{\sim} \\
&& \DGB\DGB^c(\DGN^*(\DGOmega_{\sharp}^*(\ROp_{\bullet}))) \\
\DGOmega_{\sharp}^*(\POp)\ar[uu]^{\sim}\ar@{.>}[rr]^{(1)} && \DGN^*(\DGOmega_{\sharp}^*(\ROp_{\bullet}))\ar[u]^{(7)}_-{\sim} }.
\end{equation*}
The vertical weak-equivalences on the left-hand side of this diagram are given by the unit of the bar-cobar adjunction of operads,
as well as morphism (7). The weak-equivalences (4-5) are given by a natural comparison zigzag
of totalizations, and the weak-equivalence (6) is given by the result of Theorem~\ref{thm:cooperad totalization}.

The result of Proposition~\ref{prop:operadic James construction model comparison}
implies that the vertical comparison morphism (3)
on the right-hand side of our diagram~(\ref{eq:comparison map resolution})
is a weak-equivalence too, because the cotriple resolution $\Res_*(\POp) = \FreeOp_*^{\bullet+1}(\POp)$
is given by an operadic James construction dimensionwise $\Res_*(\POp) = \FreeOp_*(\MOp)$, with $\MOp = \FreeOp_*^{\bullet}(\POp)$,
and the assumptions of our theorem implies that this generating symmetric sequence $\MOp = \FreeOp_*^{\bullet}(\POp)$
satisfies the requirements of Proposition~\ref{prop:operadic James construction model comparison}.
Then we just use the commutativity of our diagram~(\ref{eq:comparison map resolution})
to conclude that the comparison morphism $\chi: \DGOmega^*_{\sharp}(\POp)(r)\rightarrow\DGOmega^*(\POp(r))$
is a weak-equivalence as well, for each arity $r>0$. The proof of Theorem~\ref{thm:operadic Sullivan model validity} is now complete.
\end{proof}

To complete the results of this section, we record the following generalization
of an observation given in~\cite[Proposition II.10.1.4]{FresseBook}:

\begin{prop}\label{prop:operadic Sullivan model universal interpretation}
Let $\POp\in\Simp\Op_{\varnothing}$. Let $\AOp\in\dg^*\Hopf\Op_{0}^c$.
We have a natural one-to-one correspondence between the morphisms of Hopf cochain dg-cooperads
\begin{equation*}
\phi_{\sharp}: \AOp\rightarrow\DGOmega_{\sharp}^*(\POp)
\end{equation*}
and the collections of morphisms of unitary commutative cochain dg-algebras
\begin{equation*}
\phi: \AOp(r)\rightarrow\DGOmega^*(\POp(r)),\quad r>0,
\end{equation*}
which preserve the actions of symmetric groups on our objects, preserve the counit splitting of the components of our cooperads in arity one,
and make the following diagram commute
\begin{equation*}
\xymatrix{ \AOp(k+l-1)\ar[rrr]^-{\phi}\ar[d]_{\circ_{i}^{*}} &&& \DGOmega^*(\POp(k+l-1))\ar[d]_{(\circ_{i})^{*}} \\
\AOp(k)\otimes\AOp(l)\ar[rr]^-{\phi\otimes\phi} && \DGOmega^*(\POp(k))\otimes\DGOmega^*(\POp(l))\ar[r]^-{\nabla} &
\DGOmega^*(\POp(k)\times\POp(l)) }
\end{equation*}
for each $k,l>0$, for any $i\in\{1<\dots<k\}$, and where $\nabla$ refers to the natural codiagonal morphism
associated to the functor $\DGOmega^*: \Simp^{op}\rightarrow\dg^*\ComCat_+$.
\end{prop}

\begin{proof}
We use the same correspondence as in~\cite[Proposition II.10.1.4]{FresseBook}.
In a first step, we use the adjunction relation $\DGG: \dg^*\Hopf\Op_0^c\rightleftarrows\Simp\Op_{\varnothing}^{op} :\DGOmega^*_{\sharp}$
in order to associate a morphism of operads in simplicial sets $\phi_{\flat}: \POp\rightarrow\DGG(\AOp)$
to the morphism of Hopf cochain dg-cooperads $\phi_{\sharp}: \AOp\rightarrow\DGOmega_{\sharp}^*(\POp)$.
Then we use the adjunction relation $\DGG: \dg^*\ComCat\rightleftarrows\Simp^{op} :\DGOmega^*$
to associate morphisms of commutative cochain dg-algebras $\phi: \AOp(r)\rightarrow\DGOmega^*(\POp(r))$
to the morphisms of simplicial sets $\phi_{\flat}: \POp(r)\rightarrow\DGG(\AOp(r))$
which underlie this operad morphism $\phi_{\flat}: \POp\rightarrow\DGG(\AOp)$.
We easily check that the preservation of the operadic structure operations for the morphisms of simplicial sets $\phi_{\flat}: \POp(r)\rightarrow\DGG(\AOp(r))$
is equivalent to the requirements of the theorem for the morphisms
of commutative cochain dg-algebras $\phi: \AOp(r)\rightarrow\DGOmega^*(\POp(r))$.
\end{proof}

\section{The rational homotopy theory of operads}\label{sec:operad rational homotopy}
We explain the consequences of the results of the previous section for the study of the rational homotopy of operads in this section.
We actually get the same results as in~\cite[\S II.10.2]{FresseBook} in the context of operads
that reduce to a one-point set in arity one.
We therefore only give a brief overview of the main statements.

Recall that the Sullivan rationalization of a simplicial set $X\in\Simp$
is defined by $X^{\QQ} = \DGL\DGG(\DGOmega^*(X))$,
where $\DGL\DGG$ denotes the left derived functor of the functor $\DGG: \dg^*\ComCat\rightarrow\Simp^{op}$
on the category of commutative cochain dg-algebras $\dg^*\ComCat$.
Let $\POp$ be a cofibrant object in the category of operads in simplicial sets $\Simp\Op_{\varnothing}$.
We set $\POp^{\QQ} = \DGL\DGG(\DGOmega^*_{\sharp}(\POp))$,
where $\DGL\DGG$ now denotes the left derived functor of the functor $\DGG: \dg^*\Hopf\Op_0^c\rightarrow\Simp\Op_{\varnothing}^{op}$
on the category of Hopf cochain dg-cooperads $\dg^*\Hopf\Op_0^c$.
We explicitly have $\POp^{\QQ} = \DGG(\AOp)$, where $\AOp\xrightarrow{\sim}\DGOmega^*_{\sharp}(\POp)$
is any cofibrant resolution of the object $\DGOmega^*_{\sharp}(\POp)$
in our model category of Hopf cochain dg-cooperads $\dg^*\Hopf\Op_0^c$.
We have the following result:

\begin{thm}\label{thm:operad rationalization}
We assume that the operad $\POp$ consists of connected simplicial sets $\POp(r)$
whose homology with rational coefficients $\DGH_*(\POp(r)) = \DGH_*(\POp(r),\QQ)$
form modules of finite dimension over the ground field $\QQ$
in each degree.
We then have a weak-equivalence in the category of simplicial sets $\POp^{\QQ}(r)\sim\POp(r)^{\QQ}$
in each arity $r>0$, where $\POp(r)^{\QQ}$
denotes the Sullivan rationalization of the simplicial set $X = \POp(r)$.
\end{thm}

\begin{proof}
We use that the collection of commutative cochain dg-algebras which underlie a cofibrant Hopf cochain dg-cooperad consist of cofibrant objects
in the category of commutative cochain dg-algebras (see~\ref{prop:Hopf cooperad cofibrations}).
We deduce from this result that we have a weak-equivalence $\DGL\DGG(\AOp)(r)\sim\DGL\DGG(\AOp(r))$
in the category of commutative cochain dg-algebras in each arity $r>0$,
for any Hopf cochain dg-cooperad $\AOp$,
where we consider the component of arity $r$ of the operad $\DGL\DGG(\AOp)\in\Simp\Op_{\varnothing}$
on the left-hand side, and the simplicial set $\DGL\DGG(\AOp(r))\in\Simp$
associated to commutative cochain dg-algebra $\AOp(r)$
on the right-hand side.
We have on the other hand $\DGL\DGG(\DGOmega^*_{\sharp}(\POp)(r))\sim\DGL\DGG(\DGOmega^*(\POp(r)))$,
since the weak-equivalence $\DGOmega^*_{\sharp}(\POp)(r)\sim\DGOmega^*(\POp(r))$
given by the result of Theorem~\ref{thm:operadic Sullivan model validity}
induces a weak-equivalence of simplicial sets when we apply our derived functor $\DGL\DGG$
on the category of commutative cochain dg-algebra. The conclusion follows.
\end{proof}

For our purpose, we also record the following immediate follow-up of our constructions:

\begin{thm}\label{thm:rational operad mapping spaces}
Let $\POp,\QOp\in\Simp\Op_{\varnothing}$. We have a weak-equivalence of simplicial sets:
\begin{equation*}
\Map^h(\POp,\QOp^{\QQ})\sim\Map^h(\DGR\DGOmega^*_{\sharp}(\QOp),\DGR\DGOmega^*_{\sharp}(\POp)),
\end{equation*}
where, on the left-hand side, we consider the derived mapping space $\Map^h(-,-)$ associated to the objects $\POp,\QOp^{\QQ}\in\Simp\Op_{\varnothing}$
in the model category of operads in simplicial sets, and on the right-hand side, we consider the derived functor $\DGR\DGOmega^*_{\sharp}(-)$
of our operadic upgrading of the Sullivan cochain dg-algebra functor $\DGOmega^*_{\sharp}(-)$
and the derived mapping space associated to the objects $\DGR\DGOmega^*_{\sharp}(\POp),\DGR\DGOmega^*_{\sharp}(\QOp)\in\dg^*\Hopf\Op_0^c$
in the model category of Hopf cochain dg-cooperads.
\end{thm}

\begin{proof}
This theorem is an immediate consequence
of the Quillen adjunction relation $\DGG: \dg^*\Hopf\Op_0^c\rightleftarrows\Simp\Op_{\varnothing}^{op} :\DGOmega^*_{\sharp}$.
\end{proof}

\section{The extension to unitary operads}\label{sec:unitary operad rational homotopy}
Recall that we use the notation $\Lambda$ for the category which has the finite ordinals $\rset = \{1<\dots,r\}$
and all injective maps (not necessarily monotoneous) between such ordinals as morphisms.
In what follows, we also consider the full subcategory $\Lambda_{>0}\subset\Lambda$
generated by the ordinals $\rset = \{1<\dots,r\}$
such that $r>0$.

Recall also that the structure of a $\Lambda$-operad in simplicial sets consists of an operad $\POp$
equipped with restriction operators $u^*: \POp(l)\rightarrow\POp(k)$,
associated to the injective maps $u: \{1<\dots<k\}\rightarrow\{1<\dots<l\}$,
so that the action of the symmetric groups on the objects $\POp(r)$, $r>0$,
extends to an action of the category $\Lambda_{>0}$
on our collection.
We assume that the composition products of the operad $\POp$ satisfy natural equivariance relations,
which extend the usual equivariance axioms of operads,
with respect to the action of restriction operators. These equivariance relations involve, in degenerate cases,
the action of an augmentation on $\POp(r)$,
which is just the constant map to the one-point set in our context $\epsilon: \POp(r)\rightarrow\pt$.
Thus, we do not specify these augmentation maps in our terminology but more care would be necessary when we consider $\Lambda$-operads
in other categories than the category of simplicial sets (or a cartesian category).
We use the notation $\Simp\Lambda\Op_{\varnothing}$ for the category of $\Lambda$-operads in simplicial sets.

In~\cite[pp. xix-xxii]{FresseBook}, we explain that the structure of a $\Lambda$-operad $\POp$
is equivalent to the structure of an (ordinary) symmetric operad $\POp_+$
such that $\POp_+(0) = *$ and $\POp_+(r) = \POp(r)$ for $r>0$.
The restriction operators $u^*: \POp(l)\rightarrow\POp(k)$ correspond to composites with the arity zero element $*\in\POp_+(0)$
at positions $j\not\in\{u(1),\dots,u(k)\}$ in the operad $\POp_+$.
To be explicit, if we use an expression in terms of variables, then for an element $p = p(x_1,\dots,x_l)\in\POp(l)$,
we have the identity $u^* p = p(y_1,\dots,y_l)$ with $y_j = x_{u^{-1}(j)}$ when $j\in\{u(1),\dots,u(k)\}$
and $y_j = *$ otherwise.
The augmentation map $\epsilon: \POp(r)\rightarrow\pt$ (which is just trivial in our context)
can be identified with the composition operation $\epsilon(p) = p(*,\dots,*)$
where we put the arity zero element $*\in\POp_+(0)$
at all positions.
This correspondence implies that we have an isomorphism of categories between the category of $\Lambda$-operads $\Simp\Lambda\Op_{\varnothing}$
and the category $\Op_*$, whose objects are the ordinary symmetric operads
such that $\POp_+(0) = *$.

In~\cite[\S II.8.4]{FresseBook}, we also explain that the category of $\Lambda$-operads $\Simp\Lambda\Op_{\varnothing}$
inherits a model structure defined by transporting the Reedy model structure
on the category of contravariant $\Lambda_{>0}$-diagrams
to the category of $\Lambda$-operads. We aim to extend the definition of the Quillen adjunction
of~\S\ref{sec:operadic Sullivan model}
to this model category. In a preliminary step, we explain the definition of the counterpart
of $\Lambda$-structures on the category of cochain dg-cooperads
and on the category of Hopf cochain dg-cooperads.
In fact, we use a straightforward extension of the definitions of the book~\cite{FresseBook}
by using the extension of the notion of a cooperad which we consider in~\S\ref{cooperads:general definition}.
Therefore, we only review the main features of these definitions, and we refer to~\cite[\S\S II.11-12]{FresseBook}
for more details on the relations which we use in our constructions.
In passing, we also review the definition of model structures on these extensions of the category of cochain dg-cooperads
and of the category of Hopf cochain dg-cooperads.

\subsubsection{The model category of cochain dg-$\Lambda$-cooperads}\label{Lambda-cooperads:model structure}
For our purpose, a cochain dg-$\Lambda$-cooperad consists of a (conilpotent) cooperad in cochain graded dg-modules $\COp$,
in the sense of~\S\ref{cooperads:general definition},
together with corestriction operators $u_*: \COp(k)\rightarrow\COp(l)$,
associated to the injective maps $u: \{1<\dots<k\}\rightarrow\{1<\dots<l\}$
and which gives a covariant action of the category $\Lambda_{>0}$
on the collection $\COp(r)$, as well as coaugmentation morphisms $\epsilon^*: \QQ\rightarrow\COp(r)$,
defined for all $r\in\NN$. We assume that these structure operations satisfy an obvious extension of the requirements
of~\cite[\S II.11.1.1]{FresseBook}.
To be more precise, in our context where we do not necessarily assume that $\COp(1)$
reduces to the ground field $\QQ$,
we just make the additional requirement that the component of the coaugmentation of arity one $\epsilon^*: \QQ\rightarrow\COp(1)$
is equal to the section of the counit $\eta^*: \COp(1)\rightarrow\QQ$
given with the conilpotent structure
of our cooperad $\COp$. In~\cite[Proposition II.11.1.4]{FresseBook}, we observe that the coaugmentation morphisms $\epsilon^*: \QQ\rightarrow\COp(r)$
are identified with a morphism of cooperads $\epsilon^*: \ComOp^c\rightarrow\COp$,
where $\ComOp^c$ denotes the dual cooperad of the operad of commutative algebras, which is given by the ground field $\ComOp^c(r) = \QQ$
in each arity $r>0$. Thus, every cochain dg-$\Lambda$-cooperad is naturally coaugmented over the commutative cooperad $\ComOp^c$.
In what follows, we use the notation $\ComOp^c/\dg^*\Lambda\Op_0^c$ for this category of (coaugmented) cochain dg-$\Lambda$-cooperads
defined in this paragraph.
In~\cite[\S II.11]{FresseBook},
we use the full expression `coaugmented $\Lambda$-cooperad' for the objects of our category of $\Lambda$-cooperads,
where the adjective `coaugmented' refers to this part of the structure of a $\Lambda$-cooperad
which is given by this coaugmentations over the unit object $\ComOp^c(r) = \QQ$.
To simplify our terminology, we prefer to drop this adjective from the name
given to the category of $\Lambda$-cooperads in this paper, and therefore we just call cochain dg-$\Lambda$-cooperads
the objects of the category of $\Lambda$-cooperads
defined in this paragraph $\ComOp^c/\dg^*\Lambda\Op_0^c$.
Let us mention that we use an analogue of this category of $\Lambda$-cooperads
in the category of general (unbounded) dg-modules
as an auxiliary category in the paper~\cite{FresseTurchinWillwacher}.

We have an obvious forgetful functor $\omega: \ComOp^c/\dg^*\Lambda\Op_0^c\rightarrow\ComOp^c/\dg^*\Op_0^c$,
where we consider the category of coaugmented objects $\ComOp^c/\dg^*\Op_0^c$
in the category of ordinary cochain dg-cooperads $\dg^*\Op_0^c$.
This functor admits a left adjoint $\ComOp^c/\Lambda\otimes_{\Sigma}-: \ComOp^c/\dg^*\Op_0^c\rightarrow\ComOp^c/\dg^*\Lambda\Op_0^c$
which is given by a straightforward extension of the construction of~\cite[\S II.11.2]{FresseBook}.
Let us simply mention that the expression $\Lambda\otimes_{\Sigma}-$ in the notation of this left adjoint functor refers to the fact
that we use a Kan extension construction to extend the action of the symmetric groups
on a coaugmented ordinary cooperads $\COp\in\ComOp^c/\dg^*\Op_0^c$
to an action of the category $\Lambda$.
We use this adjunction $\ComOp^c/\Lambda\otimes_{\Sigma}-: \ComOp^c/\dg^*\Op_0^c\rightleftarrows\ComOp^c/\dg^*\Lambda\Op_0^c :\omega$
to transport the model structure of the undercategory of cochain dg-cooperads $\ComOp^c/\dg^*\Op_0^c$
to the category of cochain dg-$\Lambda$-cooperads.
We explicitly define the class of weak-equivalences (respectively, the class of fibrations)
in the category of cochain dg-$\Lambda$-cooperads as the class of morphisms
which form a weak-equivalence (respectively, a fibration)
in the category of cochain dg-cooperads. (In particular, we still get that a morphism of cochain dg-$\Lambda$-cooperads
is a weak-equivalence $\phi: \COp\xrightarrow{\sim}\DOp$ if this morphism induces an isomorphism in cohomology.)
We define the class of cofibrations of cochain dg-$\Lambda$-cooperads as the class of morphisms
which have the left lifting property with respect to our class of acyclic fibrations.
This model category is cofibrantly generated
with the morphisms $\ComOp^c/\Lambda\otimes_{\Sigma}i: \ComOp^c/\Lambda\otimes_{\Sigma}\COp\rightarrow\ComOp^c/\Lambda\otimes_{\Sigma}\DOp$
associated to the generating (acyclic) cofibrations
of the undercategory of cochain dg-cooperads $i: \COp\rightarrow\DOp$
as a set of generating (acyclic) cofibrations. The proof of the validity of the definition of this cofibrantly generated model structure
on our category of cochain dg-$\Lambda$-cooperads $\dg^*\Lambda\Op_0^c$ follows from an immediate extension of the verifications
of~\cite[\S II.11.3]{FresseBook} in the context of the subcategory of cochain dg-$\Lambda$-cooperads
whose term of arity one reduces to the ground field.

\subsubsection{The model category of Hopf cochain dg-$\Lambda$-cooperads}\label{Hopf Lambda-cooperads:model structure}
We define our category of Hopf cochain dg-$\Lambda$-cooperads by an obvious extension,
in the category of commutative cochain dg-coalgebras,
of our notion of a cochain dg-$\Lambda$-cooperad. Thus, we assume that a Hopf cochain dg-$\Lambda$-cooperad $\AOp$
is a Hopf cochain dg-cooperad in the sense of~\S\ref{Hopf cooperads:model structure}
equipped with corestriction operators $u_*: \AOp(k)\rightarrow\AOp(l)$
and coaugmentation morphisms $\epsilon^*: \QQ\rightarrow\AOp(r)$,
defined in the category of commutative cochain dg-coalgebras,
and which satisfy the structure relations of the definition of a $\Lambda$-cooperad
in this category. Note that, in this context, the coaugmentation morphisms $\epsilon^*: \QQ\rightarrow\AOp(r)$
are necessarily identified with the unit morphisms of the commutative cochain dg-algebras $\AOp(r)$
which define our Hopf cooperad $\AOp$. Therefore, we usually omit to specify these structure
operations in our definitions. We use the notation $\dg^*\Hopf\Lambda\Op_0^c$
for this category of Hopf cochain dg-$\Lambda$-cooperads which we use in this paper. Note again that, in comparison with the definition
of~\cite[\S II.11.4.1]{FresseBook}, we just do not necessarily assume that the term of arity one
of our object reduces to the unit algebra $\QQ$. We also use an analogue of this category of Hopf $\Lambda$-cooperads
in the category of general (unbounded) dg-modules
in the paper~\cite{FresseTurchinWillwacher}.

We have a commutative square of forgetful and adjoint functors that relate the category of Hopf cochain dg-$\Lambda$-cooperads $\dg^*\Hopf\Lambda\Op_0^c$
to the category of ordinary Hopf cochain dg-cooperads $\dg^*\Hopf\Op_0^c$
and to the category of cochain dg-$\Lambda$-cooperads $\ComCat^c/\dg^*\Lambda\Op_0^c$:
\begin{equation}\tag{$*$}\label{eq:Hopf Lambda cooperad adjunctions}
\xymatrix{ \ComCat^c/\dg^*\Lambda\Op_0^c\ar@<+2pt>@{.>}[rr]^{\ComOp^c/\Sym(-)}\ar@<-2pt>[dd] &&
\dg^*\Hopf\Lambda\Op_0^c\ar@<+2pt>[ll]\ar@<-2pt>[dd] \\
&& \\
\ComCat^c/\dg^*\Op_0^c\ar@<+2pt>@{.>}[rr]^{\ComOp^c/\Sym(-)}\ar@<-2pt>@{.>}[uu]_{\ComOp^c/\Lambda\otimes_{\Sigma}-} &&
\dg^*\Hopf\Op_0^c\ar@<+2pt>[ll]\ar@<-2pt>@{.>}[uu]_{\ComOp^c/\Lambda\otimes_{\Sigma}-} }.
\end{equation}
We use solid arrows to materialize the obvious forgetful functors of this diagram
and dotted arrows to depict the corresponding adjoint functors.
We already explained that the functor $\ComOp^c/\Lambda\otimes_{\Sigma}-: \ComCat^c/\dg^*\Op_0^c\rightarrow\ComCat^c/\dg^*\Lambda\Op_0^c$,
left adjoint to the forgetful functor from $\Lambda$-cooperads to coaugmented cooperads,
is given by a Kan extension process. We use the same categorical construction, within the category of Hopf symmetric sequences,
to define the left adjoint of the forgetful functor from Hopf $\Lambda$-cooperads
to Hopf cooperads $\ComOp^c/\Lambda\otimes_{\Sigma}-: \dg^*\Hopf\Op_0^c\rightarrow\ComCat^c/\dg^*\Hopf\Lambda\Op_0^c$.
The functors $\ComOp^c/\Sym(-)$ are given by a relative analogue
of the operadic symmetric algebra functor $\Sym_{\Op_0^c}(-): \dg^*\Op_0^c\rightarrow\dg^*\Hopf\Op_0^c$
of~\S\ref{Hopf cooperads:model structure}.
To be explicit, to a coaugmented cooperad $\COp\in\ComOp^c/\dg^*\Op_0^c$,
we associate the Hopf cooperad such that $\ComOp^c/\Sym(\COp)(r) = \Sym(\COp(r))\otimes_{\Sym(\QQ)}\QQ$,
for each arity $r>0$,
where we again use a relative tensor product to identify the image of the morphism $\Sym(\epsilon^*): \Sym(\QQ)\rightarrow\Sym(\COp(r))$
induced by the coaugmentation of our object $\QQ = \ComOp^c(r)\rightarrow\COp(r)$
with the unit of the symmetric algebra. We just check that this Hopf cooperad $\ComOp^c/\Sym(\COp)(r) = \Sym(\COp(r))\otimes_{\Sym(\QQ)}\QQ$
inherits a natural Hopf $\Lambda$-cooperad structure when we assume $\COp\in\ComOp^c/\dg^*\Lambda\Op_0^c$
(see~\cite[Proposition II.11.4.5]{FresseBook} for an analogue of this construction in the context of cooperads
whose term of arity one reduces to the unit object).

We use the adjunction $\ComOp^c/\Lambda\otimes_{\Sigma}-: \dg^*\Hopf\Op_0^c\rightleftarrows\dg^*\Hopf\Lambda\Op_0^c :\omega$
to transport the model structure of the category of Hopf cochain dg-cooperads
to the category of Hopf cochain dg-$\Lambda$-cooperads.
We explicitly define the class of weak-equivalences (respectively, the class of fibrations)
in the category of Hopf cochain dg-$\Lambda$-cooperads as the class of morphisms
which form a weak-equivalence (respectively, a fibration)
in the category of Hopf cochain dg-cooperads. (In particular, we still get that a morphism of Hopf cochain dg-$\Lambda$-cooperads
is a weak-equivalence $\phi: \AOp\xrightarrow{\sim}\BOp$ if this morphism induces an isomorphism in cohomology.)
We define the class of cofibrations of Hopf cochain dg-$\Lambda$-cooperads as the class of morphisms
which have the left lifting property with respect to the class of acyclic fibrations
determined by this definition.
This model category is cofibrantly generated again
with the morphism $\ComOp^c/\Lambda\otimes_{\Sigma}\phi: \ComOp^c/\Lambda\otimes_{\Sigma}\KOp\rightarrow\ComOp^c/\Lambda\otimes_{\Sigma}\LOp$
associated to the generating (acyclic) cofibrations
of the category of Hopf cochain dg-$\Lambda$-cooperads $\phi: \KOp\rightarrow\LOp$
as class of generating (acyclic) cofibrations. The proof of the validity of the definition of this cofibrantly generated model structure
on our category of Hopf cochain dg-$\Lambda$-cooperads $\dg^*\Hopf\Lambda\Op_0^c$ follows again from an immediate extension
of the verifications of~\cite[\S II.11.4]{FresseBook}.
We can also check that all adjunctions in our commutative square (\ref{eq:Hopf Lambda cooperad adjunctions})
are Quillen adjunctions.

\medskip
We readily check that the arity-wise image $\DGG(\AOp)(r) = \Mor_{\dg^*\ComCat}(\AOp(r),\DGOmega^*(\Delta^{\bullet}))$
of a Hopf cochain dg-$\Lambda$-cooperad $\AOp\in\dg^*\Hopf\Lambda\Op_0^c$
under the left adjoint of the Sullivan cochain dg-algebra functor $\DGG: \dg^*\ComCat\rightarrow\Simp^{op}$
inherits a natural $\Lambda$-operad structure.
We accordingly get that the functor $\DGG: \dg^*\Hopf\Op_0^c\rightarrow\Simp\Op_{\varnothing}^{op}$
of~\S\ref{sec:operadic Sullivan model}
admits an extension to the category of $\Lambda$-operads:
\begin{equation*}
\DGG: \dg^*\Hopf\Lambda\Op_0^c\rightarrow\Simp\Lambda\Op_{\varnothing}^{op}.
\end{equation*}
We also have the following statement:

\begin{thm}\label{thm:Lambda-operad Sullivan model}
The functor $\DGOmega^*_{\sharp}: \POp\mapsto\DGOmega^*_{\sharp}(\POp)$ of~\S\ref{operadic Sullivan model:construction}
lifts to a functor:
\begin{equation*}
\DGOmega^*_{\sharp}: \Simp\Lambda\Op_{\varnothing}^{op}\rightarrow\dg^*\Hopf\Lambda\Op_0^c
\end{equation*}
adjoint to the functor $\DGG: \dg^*\Hopf\Lambda\Op_0^c\rightarrow\Simp\Lambda\Op_{\varnothing}^{op}$
from the category of Hopf cochain dg-$\Lambda$-cooperads $\dg^*\Hopf\Lambda\Op_0^c$
to the category of $\Lambda$-operads $\Simp\Lambda\Op_{\varnothing}$.
\end{thm}

\begin{proof}
We again use the adjoint functor lifting theorem to define a right adjoint
of the functor $\DGG: \dg^*\Hopf\Lambda\Op_0^c\rightarrow\Simp\Lambda\Op_{\varnothing}^{op}$.
We have in particular $\DGOmega^*_{\sharp}(\FreeOp(\MOp)) = \FreeOp^c(\DGOmega^*(\MOp))$
when $\POp = \FreeOp(\MOp)$ is a free object of the category of $\Lambda$-operads.
Recall simply that the forgetful functor from $\Lambda$-operads to ordinary symmetric operads creates these free objects (see~\cite[\S II.8.4]{FresseBook}).
We have an analogous result for the cofree objects of the category of (Hopf) cochain dg-$\Lambda$-cooperads.
These observations imply that our functor $\DGOmega^*_{\sharp}: \Simp\Lambda\Op_{\varnothing}^{op}\rightarrow\dg^*\Hopf\Lambda\Op_0^c$
reduces to the functor $\DGOmega^*_{\sharp}: \Simp\Op_{\varnothing}^{op}\rightarrow\dg^*\Hopf\Op_0^c$
of~\S\ref{operadic Sullivan model:construction}
on the category of ordinary symmetric operads, and hence, defines a lifting to the category of $\Lambda$-operads
of this previously defined operadic upgrading of the Sullivan dg-algebra functor on simplicial sets,
as stated in our theorem.
\end{proof}

Let us mention that we have an obvious extension of the result of Proposition~\ref{prop:operadic Sullivan model universal interpretation}
to the category of $\Lambda$-operads.
In this setting, we just get that the morphisms $\phi: \AOp(r)\rightarrow\DGOmega^*(\POp(r))$
which correspond to a morphism of Hopf cochain dg-$\Lambda$-cooperads $\phi_{\sharp}: \AOp\rightarrow\DGOmega^*_{\sharp}(\POp)$
also preserve the action of the corestriction operators
on our objects.

\medskip
We equip the category $\Simp\Lambda\Op_{\varnothing}$ with the Reedy model structure of \cite[\S II.8.4]{FresseBook}
and the category $\dg^*\Hopf\Lambda\Op_0^c$ with the model structure of~\S\ref{Hopf Lambda-cooperads:model structure}.
We have the following observation:

\begin{prop}\label{prop:Lambda-operad Sullivan model adjunction}
The adjunction $\DGG: \dg^*\Hopf\Lambda\Op_0^c\rightleftarrows\Simp\Lambda\Op_{\varnothing} :\DGOmega^*_{\sharp}$
defines a Quillen adjunction.
\end{prop}

\begin{proof}
We use that the forgetful functor $\omega: \Simp\Lambda\Op_{\varnothing}\rightarrow\Simp\Op_{\varnothing}$
preserves the (acyclic) fibrations, and that the functor $\DGOmega^*_{\sharp}$ on the category ordinary symmetric operads
carries such (acyclic) fibrations to (acyclic) cofibrations in the category of Hopf cochain dg-cooperads.
We use the definition of our model structure on Hopf cochain dg-$\Lambda$-cooperads
to conclude that (acyclic) cofibrations do form (acyclic) cofibrations
in the category of Hopf cochain dg-cooperads,
and hence, to conclude that the functor $\DGOmega^*_{\sharp}: \Simp\Lambda\Op_{\varnothing}\rightarrow\dg^*\Hopf\Lambda\Op_0^c$
carries the (acyclic) fibrations of $\Lambda$-operads to (acyclic) cofibrations of Hopf cochain dg-cooperads. This proves our proposition.
\end{proof}

Thus, we can again use the adjunction $\DGG: \dg^*\Hopf\Lambda\Op_0^c\rightleftarrows\Simp\Lambda\Op_{\varnothing}^{op} :\DGOmega^*_{\sharp}$
to define a model for the rational homotopy of $\Lambda$-operads in simplicial sets,
and as a by-product, for the rational homotopy of unitary operads.
Note simply that the observation that the functor $\DGOmega^*_{\sharp}: \Simp\Lambda\Op_{\varnothing}^{op}\rightarrow\dg^*\Hopf\Lambda\Op_0^c$
reduces to the operadic upgrading of the Sullivan dg-algebra functor of~\S\ref{operadic Sullivan model:construction}
in Theorem~\ref{thm:Lambda-operad Sullivan model}
implies that the results of~\S\ref{sec:operad rational homotopy}
extend to the category of $\Lambda$-operads. In particular, if we assume that $\POp$ is a cofibrant $\Lambda$-operad in simplicial sets
associated to a unitary operad $\POp_+\in\Simp\Op_*$ which fulfills the assumptions of Theorem~\ref{thm:operad rationalization},
then we get that the object:
\begin{equation*}
\POp_+^{\QQ} = \DGL\DGG(\DGOmega^*_{\sharp}(\POp))_+,
\end{equation*}
where we now consider the left derived functor of the functor $\DGG: \dg^*\Hopf\Lambda\Op_0^c\rightarrow\Simp\Lambda\Op_{\varnothing}^{op}$,
defines a unitary operad in simplicial sets $\POp_+^{\QQ}\in\Simp\Op_*$
such that $\POp_+^{\QQ}(r) = \POp^{\QQ}(r)$
is equivalent to the Sullivan rationalization of the space $\POp_+(r) = \POp(r)$
in each arity $r>0$.

We also have a weak-equivalence of simplicial sets:
\begin{equation*}
\Map^h(\POp_+,\QOp_+^{\QQ})\sim\Map^h(\DGR\DGOmega^*_{\sharp}(\QOp),\DGR\DGOmega^*_{\sharp}(\POp)),
\end{equation*}
for each pair of unitary operads $\POp_+,\QOp_+\in\Simp\Op_*$, where on the left-hand side we consider the derived mapping space $\Map^h(-,-)$
associated to the objects $\POp_+,\QOp_+^{\QQ}\in\Simp\Op_*$ in the model category of operads in simplicial sets,
and on the right-hand side we consider the right derived functor $\DGR\DGOmega^*_{\sharp}(-)$
of our upgrading of the Sullivan dg-algebra functor $\DGOmega^*_{\sharp}(-)$
on the category of $\Lambda$-operads
and the derived mapping space associated to the objects $\DGR\DGOmega^*_{\sharp}(\POp),\DGR\DGOmega^*_{\sharp}(\QOp)\in\dg^*\Hopf\Op_0^c$
in the model category of Hopf cochain dg-$\Lambda$-cooperads. Note simply that the derived mapping spaces of objects $\POp_+,\QOp_+\in\Simp\Op_*$
computed in the category of unitary operads $\Simp\Op_*$ agree with the derived mapping spaces
associated to these objects in the category of all operads $\Op$ (with any term allowed in arity zero)
by the main result of~\cite{FresseTurchinWillwacherShort}.

\section{The applications to the framed little discs operads}\label{sec:framed little discs operads}

Recall that the framed little $n$-discs operad $\DOp_n^{fr}$
consists of the spaces of little discs embeddings $c_i: \DD^n\hookrightarrow\DD^n$, $i = 1,\dots,r$,
of the form $c_i(v) = \lambda_i\cdot\rho_i(v) + a_i$, where $\lambda_i\in\RR_{>0}$, $\rho_i\in\SO(n)$, $a_i\in\DD^n$,
and such that we have $\mathring{c}_i\cap\mathring{c}_j = \varnothing$ for all $i\not=j$,
where we use the notation $\mathring{c}_i$
for the interior of the image of the embedding $c_i$
inside the disc $\DD^n$.
In~\cite[Theorem 1.1]{KhoroshkinWillwacher}, it is proved that the cochain dg-algebras of differential forms with real coefficients
on the topological operad $\DOp_n^{fr}$ define a Hopf dg-cooperad (in some homotopy sense)
which is homotopy equivalent to the cohomology Hopf cochain dg-cooperad
of this topological operad $\DGH^*(\DOp_n^{fr}) = \DGH^*(\DOp_n^{fr},\RR)$.
We refer to~\cite[\S 3]{KhoroshkinWillwacher} for the precise definition of the notion of a homotopy Hopf dg-cooperad
used in this reference. We can use our functor $\DGOmega^*_{\sharp}(-)$
to rectify the constructions of this paper. To be precise, we then pass to real coefficients,
and we use the singular complex functor $\Sing(X) = \Mor_{\Top}(\Delta^{\bullet},X)$
to prolong our construction to topological operads.
We explicitly get that the object $\DGR\DGOmega^*_{\sharp}(\DOp_n^{fr})$, where we again consider the derived functor $\DGR\DGOmega^*_{\sharp}(-)$
of our operadic upgrading of the functor of differential forms $\DGOmega^*_{\sharp}(-)$,
is quasi-isomorphic to the cohomology Hopf cochain dg-cooperad $\DGH^*(\DOp_n^{fr}) = \DGH^*(\DOp_n^{fr},\RR)$
as a Hopf cochain dg-cooperad.
We conclude from this result that the operad $<\DGH^*(\DOp_n^{fr})> = |\DGL\DGG(\DGH^*(\DOp_n^{fr}))|$
where we consider the geometric realization functor $|-|$ and the left adjoint functor $\DGL\DGG(-)$ of the functor $\DGG(-)$
from Hopf cochain dg-cooperads to operads in simplicial sets,
defines a model for the realification of the object $\DOp_n^{fr}$
in the category of topological operads.

\bibliographystyle{plain}
\bibliography{OperadHomotopy-II}

\end{document}